\documentclass[12pt,leqno]{amsart}
\def\draft{n}
\usepackage{
  graphicx, dbnsymb, amsmath, amssymb, multicol, mathtools,
  picins, stmaryrd, xcolor, mathabx, needspace, enumitem,
  lgreek, sidecap, mathbbol, blindtext, bigfoot, import, amscd
}
\usepackage[all]{xy}
\usepackage[textwidth=6.5in,textheight=9in,headsep=0.15in,centering]{geometry}

\usepackage[pagebackref,breaklinks=true]{hyperref}\hypersetup{colorlinks,
  linkcolor={green!50!black},
  citecolor={green!50!black},
  urlcolor=blue
}

\usepackage{caption} 

\if\draft y
  
  \usepackage[color]{showkeys}
\fi

\def\weburl{http://drorbn.net/OU}
\def\web#1{\href{\weburl/#1}{\begin{greek}web\end{greek}/#1}}

\numberwithin{equation}{section}

\theoremstyle{plain}

\newtheorem{theorem}[equation]{Theorem}
\newtheorem{fheorem}[equation]{Fheorem}

\newtheorem{lemma}[equation]{Lemma}
\newtheorem{corollary}[equation]{Corollary}
\newtheorem{forollary}[equation]{Forollary}

\newtheorem{conjecture}[equation]{Conjecture}

\theoremstyle{definition}

\newtheorem{definition}[equation]{Definition}

\newtheorem{discussion}[equation]{Discussion}
\newtheorem{example}[equation]{Example}

\newtheorem{remark}[equation]{Remark}

\newlength{\standardunitlength}
\newcommand{\End}{\operatorname{End}}

\def\qed{{\linebreak[1]\null\hfill\text{$\Box$}}}
\def\fed{{\linebreak[1]\null\hfill\begin{picture}(0,0)%
\includegraphics{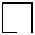}%
\end{picture}%
\setlength{\unitlength}{3710sp}%
\begingroup\makeatletter\ifx\SetFigFont\undefined%
\gdef\SetFigFont#1#2#3#4#5{%
  \reset@font\fontsize{#1}{#2pt}%
  \fontfamily{#3}\fontseries{#4}\fontshape{#5}%
  \selectfont}%
\fi\endgroup%
\begin{picture}(174,174)(-11,677)
\end{picture}%
}}

\newlength{\globalparindent}
\def\arXiv#1{{\href{http://arXiv.org/abs/#1}{arXiv:\linebreak[0]#1}}}

\def\barGamma{{\bar{\Gamma}}}
\def\bariota{{\bar{\iota}}}

\def\bbR{{\mathbb R}}

\def\calA{{\mathcal A}}
\def\calB{{\mathcal B}}

\def\calG{{\mathcal G}}

\def\calU{{\mathcal U}}
\def\calX{{\mathcal X}}
\def\calY{{\mathcal Y}}
\def\calZ{{\mathcal Z}}
\def\fraka{{\mathfrak a}}

\def\frakg{{\mathfrak g}}

\def\Hom{\operatorname{Hom}}

\def\draftcut{\if\draft y \cleardoublepage \fi}

\definecolor{lightred}{RGB}{255, 217, 217}

\def\red{\color{red}}

\def\imagetop#1{\vtop{\null\hbox{#1}}}

\def\endpar#1{~\hfill\fbox{\footnotesize#1}}

\def\calACD{\mathcal{ACD}}
\def\calAC{\mathcal{AC}}
\def\calBD{\mathcal{BD}}
\def\calB{\mathcal{B}}

\def\calOUD{\mathcal{OUD}}
\def\calOU{\mathcal{OU}}

\def\calPB{\mathcal{PB}}
\def\calROU{\mathcal{ROU}}
\def\calT{\mathcal{T}}

\def\vcalACD{v\mathcal{ACD}}
\def\vcalAC{v\mathcal{AC}}

\def\vcalOUD{v\mathcal{OUD}}
\def\vcalOU{v\mathcal{OU}}
\def\vcalPBD{v\mathcal{PBD}}
\def\vcalPB{v\mathcal{PB}}
\def\vcalROU{v\mathcal{ROU}}
\def\vcalT{v\mathcal{T}}

\def\toto{\twoheadrightarrow}

\def\Ch{\mathit{Ch}}
\def\CR{\mathit{CR}}

\begin{document} 
\setcounter{secnumdepth}{4}
\raggedbottom

\title{Over then Under Tangles}

\author{Dror~Bar-Natan}
\address{
  Department of Mathematics\\
  University of Toronto\\
  Toronto Ontario M5S 2E4\\
  Canada
}
\email{drorbn@math.toronto.edu}
\urladdr{http://www.math.toronto.edu/drorbn}

\author{Zsuzsanna Dancso}
\address{
  School of Mathematics and Statistics\\
  The University of Sydney\\
  Eastern Ave\\
  Camperdown NSW 2006, Australia
}
\email{zsuzsanna.dancso@sydney.edu.au}
\urladdr{http://zsuzsannadancso.net}

\author{Roland~van~der~Veen}
\address{
  University of Groningen, Bernoulli Institute\\
  P.O. Box 407\\
  9700 AK Groningen\\
  The Netherlands
}
\email{roland.mathematics@gmail.com}
\urladdr{http://www.rolandvdv.nl/}

\date{First edition Jul.\ 16, 2020, this edition Feb.~4,~2021. Online versions:~\cite{Self}.}


\begin{abstract} {\def\me{}
  Over-then-Under (OU) tangles are oriented tangles whose strands travel
through all of their over crossings before any under crossings. In this
paper we discuss the idea of {\em gliding\me}: an algorithm by which any
tangle diagram could be brought to OU form. Unfortunately, the algorithm
is flawed. However, by analyzing cases in which it does succeed we obtain
a braid classification result, which we also extend to virtual braids,
and provide a Mathematica implementation. We discuss other instances
of successful ``gliding ideas'' which appear in the literature --
sometimes in disguise -- such as the Drinfel'd double construction,
Enriquez's work on quantization of Lie bialgebras, and Audoux and
Meilhan's classification of welded homotopy links,

}\end{abstract}

\maketitle

\setcounter{tocdepth}{3}
\tableofcontents

\draftcut \section{Introduction}

In this paper we study {\em Over then Under} (OU) tangles,
a class of oriented tangles in which each strand travels through
all of its over-crossings before any of its under-crossings. See
Figure~\ref{fig:OUExamples} for examples; the full definition is given in Section~\ref{sec:Gliding}.

\begin{figure}
\[ 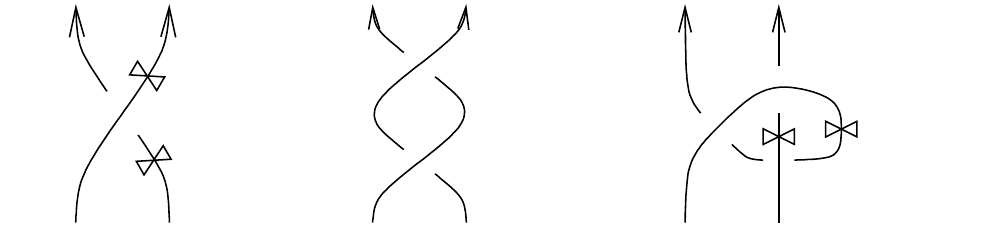\includegraphics[height=0.875in]{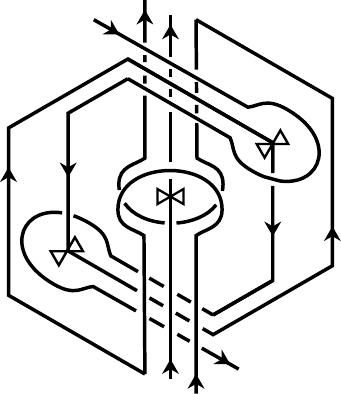} \]
\caption{The tangle diagram (A) is OU as strand 1 is all ``over''
  (so it has an empty ``U'' part) and strand 2 is all ``under'' (so it
  has an empty ``O'' part). The tangle diagram (B) is not
  OU: strand 1 is O then U, but strand 2 is U then O. Yet the tangle
  represented by (B) is OU because it is also represented
  by (C), which is OU. The diagram (D) is again OU; which familiar tangle does it represent?
} \label{fig:OUExamples}
\end{figure}

We present an algorithm (Section~\ref{sec:Gliding}) which brings non-OU
tangle diagrams to OU form using a sequence of {\em glide moves}:
specific isotopies designed to eliminate any ``forbidden sequences''
of crossings along a strand. At first glance it seems to converge for
any tangle diagram; on closer look, however, one notices that in certain
special cases of a strand crossing itself, the glide moves fail.

The goal of this paper is to investigate and review special cases where the gliding algorithm 
does converge. Indeed, when it does, it can be extremely useful, and in fact gliding ideas have been in use in knot theory
and quantum algebra for some
 time, without being recognised as part of one theme. In our opinion, there is much yet to be gained by
looking for further applications; in this paper, we present two applications in detail.

The gliding algorithm converges for braids, and every
braid -- when considered as a tangle -- has a unique OU form. Hence,
the OU form is a {\bf separating braid invariant}. We also prove that
in fact, tangles which can be brought to OU form are precisely braids,
using the identification of the braid group with the mapping class group
of a punctured disc (see Section~\ref{sec:classical}).

Even better, the gliding argument extends to virtual braids to show
that every virtual braid has a unique OU form when it is regarded
as a virtual tangle. With extra work we find that this OU form is a
{\bf complete invariant for virtual braids}. This is the subject of
Section~\ref{sec:virtual}.

Section~\ref{sec:Assorted} contains some additional comments, mostly on
the relationship between OU tangles and Hopf algebras and on ``Extraction
Graphs'', labeled graphs that are naturally associated with braids
and virtual braids by the process of recovering them from their OU forms.

In Section~\ref{sec:comp} we present Mathematica implementations, including tabulations of virtual pure braids and classical braids.

In Section~\ref{sec:more} we review a range of other instances in
the literature where ``gliding ideas'' play a role: the Drinfel'd double
construction in quantum groups, a classification of welded homotopy
links by Audoux and Meilhan  \cite{AudouxMeilhan:PeripheralSystems},
Enriquez's work on the quantization of Lie bialgebras
\cite{Enriquez:UniversalAlgebras, Enriquez:QuantizationFunctors}, and
earlier work of the authors.

All tangle diagrams in this paper are {\em open and oriented}:
Their components are always oriented intervals and never circles. For simplicity and definiteness, all
tangles in this paper are unframed: we allow all Reidemeister 1 (R1) moves, though this is not strictly
necessary and similar results also hold in the framed case.

\draftcut \section{OU Tangles and Gliding}\label{sec:Gliding}

\begin{definition} \label{def:OU} An Over-then-Under
(OU) tangle diagram
is a tangle whose strands complete all of their over crossings before any
of their under crossings, and an OU tangle is an oriented tangle that can
be represented by an OU tangle diagram. 

This is equivalent to the notion of {\em ascending} tangles
in~\cite[Definition~4.15]{AudoxBellingeriMeilhanWagner:WeldedStringLinks},
also called {\em sorted} in
\cite[Definition~1.7]{AudouxMeilhan:PeripheralSystems} in the context
of welded homotopy links.

In greater detail, an OU tangle diagram
is an oriented tangle diagram each of whose strands can be divided in
two by a ``transition point'', sometimes indicated with a bow tie symbol
$\bowtie$, such that in the first part (before the transition) it is the
``over'' strand in every crossing it goes through, and in the second part
(after the transition) it is the ``under'' strand in every crossing it
goes through, so a journey through each strand looks like an OO\ldots
O($\bowtie$)UU\ldots U sequence of crossings. Some examples are shown in
Figure~\ref{fig:OUExamples}.
\end{definition}

\begin{remark} Loosely, an
OU tangle is the ``opposite'' of an alternating tangle: crossings along each strand read OOOUUU rather than OUOUOU.
\end{remark}

Good mathematics is often discovered via wrong proofs and false theorems, which we mine for the truth still contained within. But in academic writing, one presents only the final product. Here, we take a half-page detour from academic tradition to present a Fheorem (false theorem), which, while it ultimately fails, illustrates the idea and potential of {\em gliding}. The reader who prefers tradition should rest assured that everything in the paper from Discussion~\ref{disc:froofs1} onwards is true.

\begin{fheorem}[Gliding] \label{fhm:every} Every tangle is an OU tangle.
\end{fheorem}

\par\noindent{\it Froof.} As in Figure~\ref{fig:Gliding}, the froof
is frivial. Assume first that  strands 1 and 2 are already in OU form
(meaning, all their O crossings come before all their U ones) but strand 3
still needs fixing, because at some point it goes through two crossings, first under and then over,
as on the left of Figure~\ref{fig:Gliding}. Simply glide strand 1 forward along and over 3 and
glide strand 2 back and under 3 as in Figure~\ref{fig:Gliding}, and the UO interval
along 3 is fixed, and nothing is broken on strands 1 and 2 --- strand
1 was over and remains over (more precisely, the part of strand 1 that
is shown here is the ``O'' part), and strand 2 is under and remains
under.

In fact, it doesn't matter if strands 1 and 2 are already in OU form because as shown in the
second part of Figure~\ref{fig:Gliding}, glide moves can be performed ``in bulk''. All that the
fixing of strand 3 does to strands 1 and 2 is to replace an O by an OOO on strand 1 and a U by a
UUU on strand 2, and this does not increase their complexity as UU\ldots UOO\ldots O sequences can
be fixed in one go using bulk glide moves.{\linebreak[1]\null\hfill\begin{picture}(0,0)%
\includegraphics{figs/fed.pdf}%
\end{picture}%
%
%
\setlength{\unitlength}{3710sp}%
\begingroup\makeatletter\ifx\SetFigFont\undefined%
\gdef\SetFigFont#1#2#3#4#5{%
  \reset@font\fontsize{#1}{#2pt}%
  \fontfamily{#3}\fontseries{#4}\fontshape{#5}%
  \selectfont}%
\fi\endgroup%
\begin{picture}(174,174)(-11,677)
\end{picture}%
}

\begin{figure}
\[ 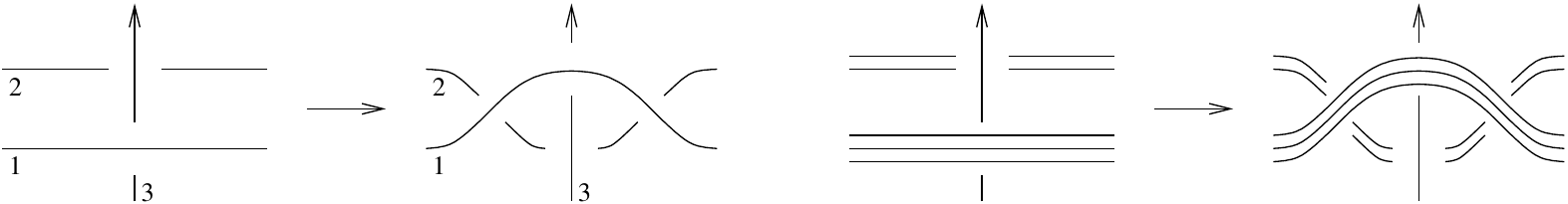 \]
\caption{Glide moves between two crossings and bulk glide moves.} \label{fig:Gliding}
\end{figure}

\parpic[r]{\begin{picture}(0,0)%
\includegraphics{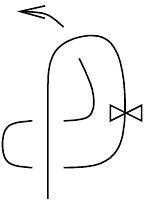}%
\end{picture}%
%
%
\setlength{\unitlength}{3947sp}%
\begingroup\makeatletter\ifx\SetFigFont\undefined%
\gdef\SetFigFont#1#2#3#4#5{%
  \reset@font\fontsize{#1}{#2pt}%
  \fontfamily{#3}\fontseries{#4}\fontshape{#5}%
  \selectfont}%
\fi\endgroup%
\begin{picture}(692,951)(-4,-148)
\end{picture}%
}
\begin{forollary}\label{for:KnotsTrivial}
All long knots are trivial.
\end{forollary}

\par\noindent{\it Froof.}
It is clear that any OU tangle on a single strand is trivial for it must
be descending as in the example on the right.
{\linebreak[1]\null\hfill\begin{picture}(0,0)%
\includegraphics{figs/fed.pdf}%
\end{picture}%
%
%
\setlength{\unitlength}{3710sp}%
\begingroup\makeatletter\ifx\SetFigFont\undefined%
\gdef\SetFigFont#1#2#3#4#5{%
  \reset@font\fontsize{#1}{#2pt}%
  \fontfamily{#3}\fontseries{#4}\fontshape{#5}%
  \selectfont}%
\fi\endgroup%
\begin{picture}(174,174)(-11,677)
\end{picture}%
}

\Needspace{20mm} 
\parpic[r]{\begin{picture}(0,0)%
\includegraphics{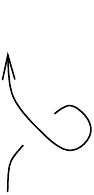}%
\end{picture}%
%
%
\setlength{\unitlength}{3947sp}%
\begingroup\makeatletter\ifx\SetFigFont\undefined%
\gdef\SetFigFont#1#2#3#4#5{%
  \reset@font\fontsize{#1}{#2pt}%
  \fontfamily{#3}\fontseries{#4}\fontshape{#5}%
  \selectfont}%
\fi\endgroup%
\begin{picture}(449,924)(-36,-73)
\end{picture}%
}
\begin{discussion} \label{disc:froofs1}
Forollary~\ref{for:KnotsTrivial} is clearly false, and so froof of the Gliding Fheorem (\ref{fhm:every}) must be false.
Indeed, while
everything we said about glide moves holds true, there is another way a strand may be U and
then O: the U and O may be parts of a single crossing, as on the right, instead of belonging to two distinct
crossings, as in the left hand side of the glide move.
\end{discussion}

\parpic[r]{\begin{picture}(0,0)%
\includegraphics{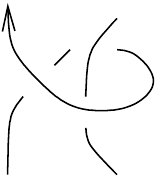}%
\end{picture}%
%
%
\setlength{\unitlength}{3947sp}%
\begingroup\makeatletter\ifx\SetFigFont\undefined%
\gdef\SetFigFont#1#2#3#4#5{%
  \reset@font\fontsize{#1}{#2pt}%
  \fontfamily{#3}\fontseries{#4}\fontshape{#5}%
  \selectfont}%
\fi\endgroup%
\begin{picture}(749,849)(-36,-223)
\end{picture}%
} It is tempting to dismiss this with
``it's only a Reidemeister 1 (R1) issue, so one may glide all
kinks to the tail of a strand and count them at the end''.  Except the
same issue can arise in ``bulk'' UU\ldots UOO\ldots O situations (as now
on the right), where it cannot be easily dismissed.  One may attempt
to resolve the UUOO situation on the right using single (non-bulk)
glide moves. We have no theoretical reason to expect this to work as the
lengths of UU\ldots U and OO\ldots O sequences may build up faster than
they are sorted. And indeed, it doesn't work. Figure~\ref{fig:swirls}
shows what happens.

\begin{figure}
\def\rO{{\red O}}\def\rU{{\red U}}\def\b{{\color{black}$\bullet$}}
\resizebox{\linewidth}{!}{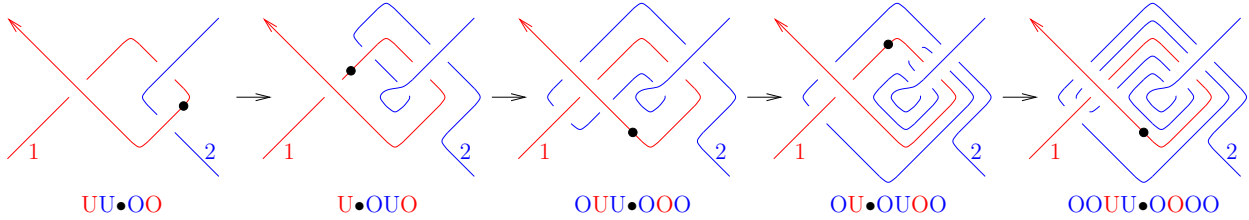}
\caption{
  An attempt to fix a non-OU tangle diagram. In each step we use a
  single glide move to fix the first UO sequence encountered on strand
  1 (we mark it with a $\bullet$), but things get progressively more
  complicated. The O/U sequences below the diagrams are listed from the
  perspective of strand 1.
} \label{fig:swirls}
\end{figure}

It is true (and also follows from Corollary~\ref{cor:cOUisBraids})
that the only 1-component OU tangle is the trivial one.
\endpar{\ref{disc:froofs1}}

\begin{discussion} \label{disc:Options}
What can we salvage from the disappointing failure of gliding?
There are many options to consider. Perhaps Fheorem~\ref{fhm:every} becomes true if
we restrict to some subset of the set of all tangles? (Braids,
Section~\ref{sec:classical}). Or perhaps if we extend to some
superset? Or in a subset of a superset? (Virtual
braids, Section~\ref{sec:virtual}). Perhaps we ought to look at some
form of finite-type completion? Perhaps we should look at tangles
in manifolds? At quotients of the space of tangles? At some
combinations of these?

In the authors' opinion it is worthwhile to explore
these options. In fact, 
many of these options have already been explored, each in a different
context and without the realization that these different contexts share
a common theme: see Section~\ref{sec:more}. \endpar{\ref{disc:Options}}
\end{discussion}

\draftcut \section{The Classical Case} \label{sec:classical}

We start with a characterization of the tangles for which the gliding
procedure of Fheorem (\ref{fhm:every}) does in fact work: in Theorem~\ref{thm:iso} we find that these are precisely braids. The following definition gets to the heart of what makes a tangle ``problematic'' for the gliding procedure:

\Needspace{36mm} 
\parpic[r]{\raisebox{0mm}{
  \begin{picture}(0,0)%
\includegraphics{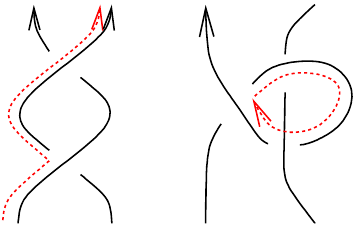}%
\end{picture}%
%
%
\setlength{\unitlength}{3947sp}%
\begingroup\makeatletter\ifx\SetFigFont\undefined%
\gdef\SetFigFont#1#2#3#4#5{%
  \reset@font\fontsize{#1}{#2pt}%
  \fontfamily{#3}\fontseries{#4}\fontshape{#5}%
  \selectfont}%
\fi\endgroup%
\begin{picture}(1702,1074)(1264,-223)
\end{picture}%
\qquad$\includegraphics[height=22mm]{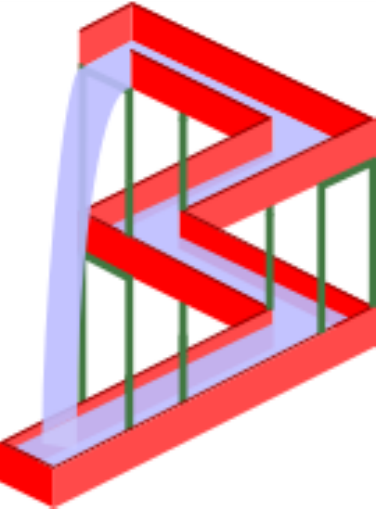}^\text{\footnotemark}$
}}
\begin{definition} Let $D$ be a tangle diagram. A ``cascade path'' along
$D$ is a directed path
that travels along strands of $D$ consistently with their orientation,
except at crossings where it can (but doesn't have to) drop from the upper
strand to the lower strand (but not the other way around). Two examples are
on the right. The diagram $D$ is called ``acyclic'' if it has no ``Escher
waterfalls'' --- that is, if no closed cascade paths can be drawn on
$D$.  On the right, the first example is
acyclic while the second isn't.
\footnotetext{Public domain waterfall image from
\url{https://commons.wikimedia.org/wiki/File:Waterfall.svg}.}
\end{definition}

\parpic[r]{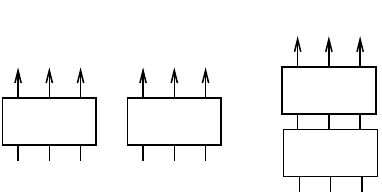}
\begin{example} \label{exa:acyclic} Braid diagrams are
acyclic tangle diagrams, and OU tangle diagrams are acyclic tangle
diagrams. The stacking product (illustrated on the right) of two acyclic
tangle diagrams is again an acyclic tangle diagram.
\end{example} 
 
Glide moves and bulk glide moves as in
Figure~\ref{fig:Gliding} do not change the acyclicity of a tangle
diagram. Indeed by simple inspection the possible transits of a cascade
path through either of the sides of a glide move are $1\to 1$, $1\to 2$,
$1\to 3$, $2\to 2$, $3\to 2$, and $3\to 3$, with numbering as in Figure~\ref{fig:Gliding}. 

Note that if a tangle diagram is OU then no Reidemeister 3 (R3) moves
can be performed on it without breaking the OU property --- if one side
of an R3 move is OU, the other necessarily isn't. This suggests that
perhaps an OU form of a tangle diagram is unique up to Reidemeister 2
(R2) moves. We aim to prove this next.

\begin{theorem} \label{thm:acyclic} A tangle diagram $D$ can be made OU
using glide moves if and only if it is acyclic, and in that case, the
resulting OU tangle diagram, which we call $\Gamma(D)$, is uniquely determined.
\end{theorem}

\begin{proof} In an acyclic tangle diagram the U and the O of a UO
interval cannot belong to the same crossing (or else an Escher waterfall
is present) so the number of UO intervals can be reduced using bulk
glide moves as in the Froof of the Gliding Fheorem (\ref{fhm:every}). By the
observation above, the resulting diagram is still acyclic so the
process can be continued. 

For the ``only if'' part, note that
OU diagrams are acyclic so anything linked to OU diagrams by glide moves
must be acyclic too.

Now to show that $\Gamma(D)$ is unique, observe that when UO intervals are apart from each
other, their fixing is clearly independent. It remains to see what
happens when UO intervals are adjacent, and there are only two
distinct cases to consider. Both of these cases are shown
in Figure~\ref{fig:unique} along with their OU fixes, which are clearly
independent of the order in which the glide moves are performed. 
\end{proof}

\begin{figure}
\[ \resizebox{0.92\linewidth}{!}{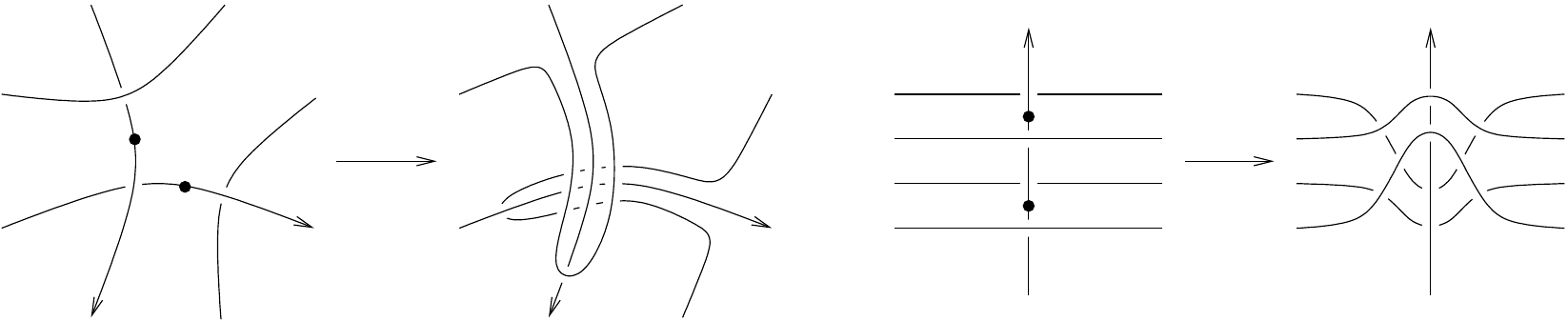} \]
\caption{Two possibilities for ``interacting'' UO intervals (each marked with a~$\bullet$ symbol).}
\label{fig:unique}
\end{figure}

\begin{corollary} The stacking product followed by $\Gamma$ makes OU tangle diagrams into a monoid.\qed
\end{corollary}

\begin{definition} A tangle diagram is called reduced if its crossing number cannot be reduced
using only R1 and R2 moves.
\end{definition}

\begin{corollary} \label{cor:GammaIso} The map $\Gamma$ descends to a well-defined map
$\barGamma$ from ``acyclic tangle diagrams modulo Reidemeister moves that
preserve the acyclic property'' into ``reduced OU tangle diagrams''.
\end{corollary}

\Needspace{34mm}
\parpic[r]{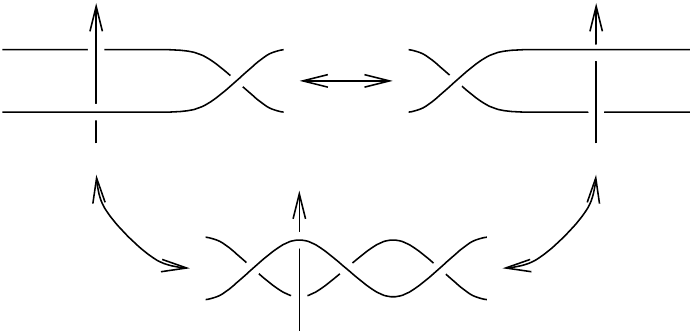}
\par\noindent{\it Proof.} If two tangle diagrams differ by an R3 move then
exactly one of them has a UO interval within the scope of the R3 move, and
its elimination via a glide move (which may as well be performed first)
yields the other diagram, up to an R2 move (picture on right). Furthermore, R1 and/or R2 moves before a glide become R1 and/or
R2 moves after the glide, or they make the glide move redundant, see
examples in Figure~\ref{fig:R1R2}. So the end result of the gliding
process of an acyclic tangle is unique modulo R1 and R2 moves. Finally it is easy
to check that within any equivalence class of acyclic tangle diagrams modulo R1
and R2 moves that preserve the acyclic property, there is a unique reduced representative. \qed

\begin{figure}
\[ 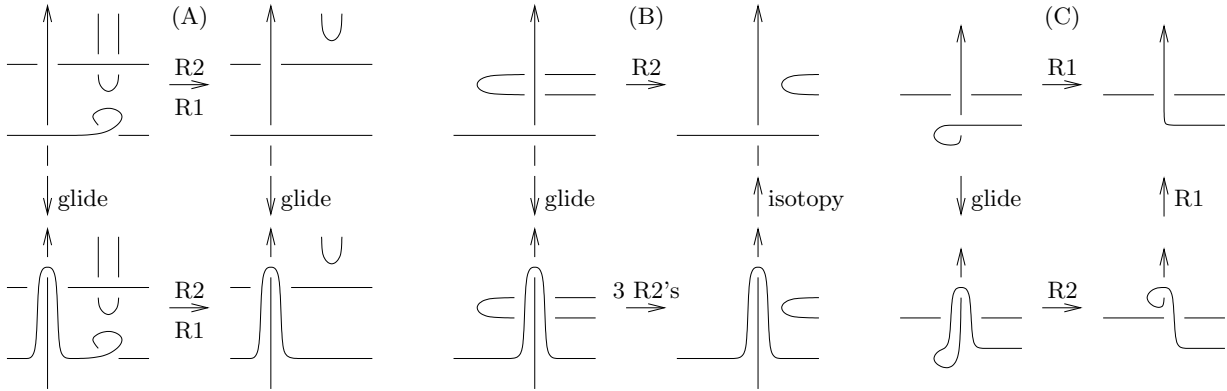 \]
\caption{R1 and R2 moves ``commute'' with glides (A), or they
make glides redundant (B), (C).} \label{fig:R1R2}
\end{figure}

\begin{corollary} \label{cor:actions}
Braids act on reduced OU tangle diagrams both on the left and on the right.
\end{corollary}

\begin{proof} Use the stacking product, the fact that braids are always acyclic, and Corollary~\ref{cor:GammaIso}.
\end{proof}

In summary, we have a commutative diagram as follows:
\[ \xymatrix@C=22mm{
  \calBD_n \ar@{^{(}->}[r]^\iota \ar[d] & \calACD_n \ar[r]^\Gamma \ar[d] & \calOUD_n \ar[d] \\
  \calB_n \ar[r]^<>(0.5)\bariota_<>(0.5){\text{Theorem~\ref{thm:iso}: }\cong} &
  \calAC_n \ar[r]^<>(0.5)\barGamma_<>(0.5){\cong} &
  \calROU_n
} \]
Here $\calBD_n$ denotes the monoid of braid diagrams with $n$ strands,
$\calACD_n$ denotes the monoid of acyclic tangle diagrams with $n$
strands, $\calOUD_n$ denotes the monoid of OU tangles diagrams with
$n$ strands, $\iota$ is the inclusion map, the vertical maps are
all ``reductions'': modulo braid moves in the first column, modulo
Reidemeister moves that preserve the acyclic property in the second
column, and modulo R1 and R2 in the third column (alternatively, the
third vertical map maps OU tangle diagrams to their unique reduced
form, and $\calROU_n$ is really a subset of $\calOU_n$), and finally,
$\bariota$ is the map induced by $\iota$ on the quotient $\calB_n$. Note
that $\barGamma$ is an isomorphism --- its inverse is the inclusion
$\calROU_n\to\calAC_n$ from Example~\ref{exa:acyclic}.

\begin{theorem}[Classical Isomorphism] \label{thm:iso}
$\barGamma\circ\bariota$ is an isomorphism (and hence also $\bariota)$.
\end{theorem}

\noindent{\it Proof.} Figure~\ref{fig:stirring} contains a visual
description of $\barGamma\circ\bariota$. If $\beta\in\calB_n$ is a braid,
to compute $\barGamma(\bariota(\beta))$ make a whisk in the shape of
$\beta$ from black metal wires, and dip it slightly into a rectangular
pool of tahini sauce. Sprinkle lines of green ground parsley on top of
the tahini pool, connecting the ends of the whisk to the front side
of the pool, as in (A) of Figure~\ref{fig:stirring}. The green tahini
lines together with the black whisk lines together still make the shape
of $\beta$, and this will remain true throughout this proof.

\begin{figure}
\resizebox{\linewidth}{!}{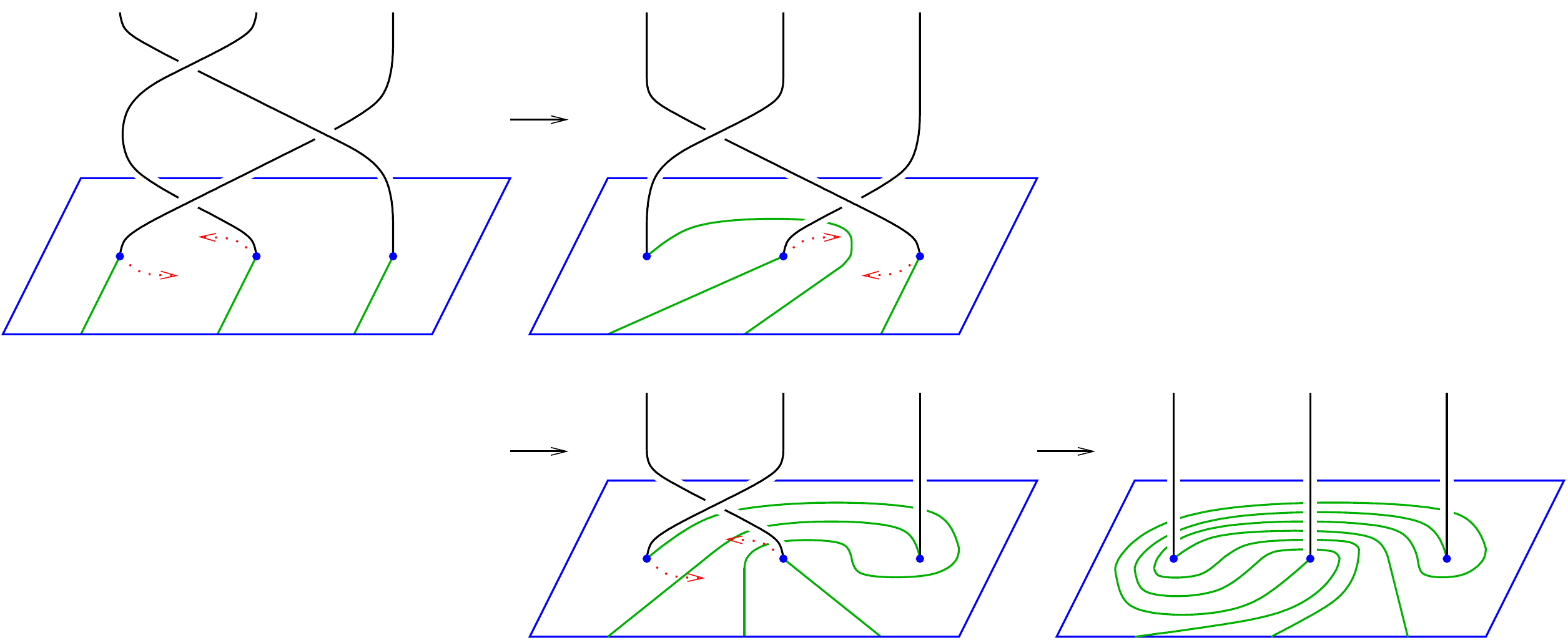}
\caption{Stirring a pool of tahini sauce garnished with parsley lines using a braid whisk.} \label{fig:stirring}
\end{figure}

Now slowly push the whisk down and let it stir the sauce as in (B), (C),
and (D) of Figure~\ref{fig:stirring}. Less and less of the whisk remains
visible and at the same time the green parsley lines remain planar but
get more and more twisty. The end of the process is in (D) and it can be
interpreted as an OU tangle, by reading the picture from top to bottom:
the black whisk wires are all O, and the green parsley lines are all
U.\footnote{Readers may recognize this as the identification
of the braid group with the mapping class group of a punctured disk. See
e.g.~\cite[Theorem~1]{BirmanBrendle:BraidsSurvey}.}

Each step of this stirring process can be broken up into glide moves
and planar equivalences that require no Reidemeister moves, as shown
in a schematic manner in Figure~\ref{fig:StirringIsGliding}. Hence our
process computes $\barGamma(\bariota(\beta))$.

\begin{figure}
\resizebox{6.25in}{!}{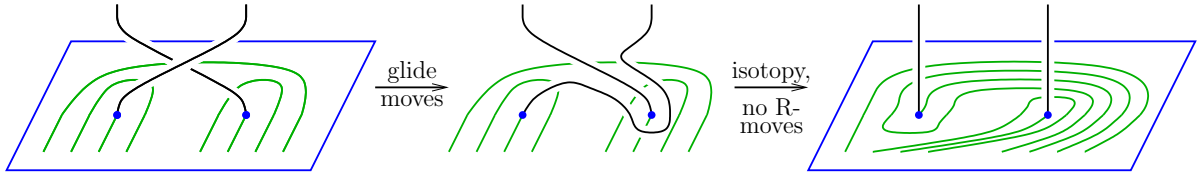}
\caption{Stirring is gliding.} \label{fig:StirringIsGliding}
\end{figure}

Every OU tangle diagram $T$ has a black-green presentation as in (D)
of Figure~\ref{fig:stirring}. Indeed the O parts of $T$ cannot cross
each other so they can be drawn as a collection of straight parallel
black lines, and the U parts do cross the O parts so perhaps they
cannot be drawn as straight lines, but they still do not cross each
other so they make a collection of ``green'' lines, leading to a
picture as in (D) of Figure~\ref{fig:stirring} or as in (A) of
Figure~\ref{fig:ReverseGamma}.

\begin{figure}
\[ \includegraphics[width=6in]{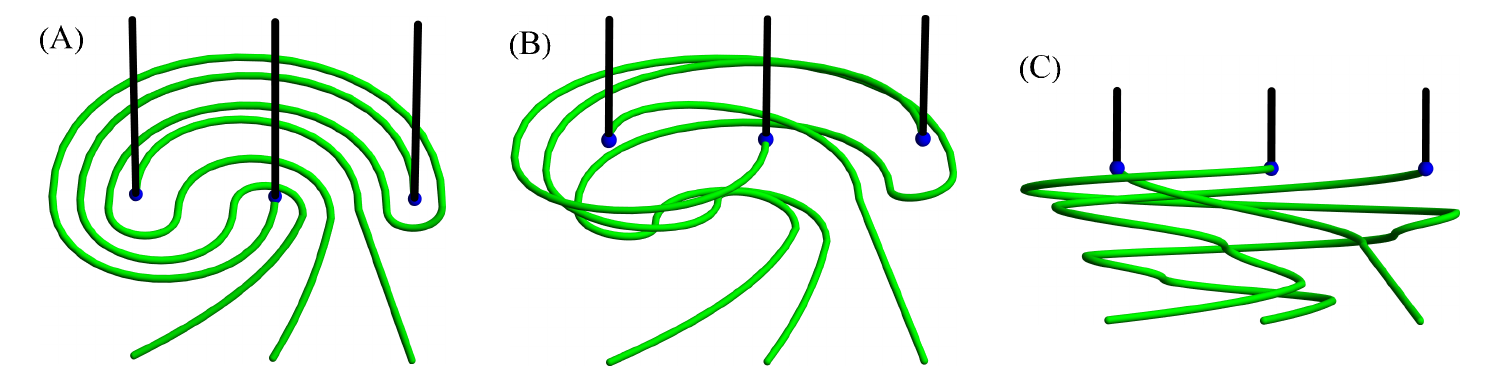} \]
\caption{The map $\Lambda$ turning an OU tangle into a braid.} \label{fig:ReverseGamma}
\end{figure}

Figure~\ref{fig:ReverseGamma} also shows how to define a map $\Lambda$
from OU tangles into braids: draw an OU tangle $T$ as in (A)
of Figure~\ref{fig:ReverseGamma}, and gradually pull down the green
strands to below the tahini level by an amount proportional to their
arc-length distance from their meeting points with the black strands,
while at the same time moving your viewpoint to be on the tahini plane,
as shown in (B) and (C) of Figure~\ref{fig:ReverseGamma}. At the end of
the process what you see is the braid $\Lambda(T)$.

Both compositions of $\barGamma\circ\bariota$ and of $\Lambda$ are identity maps\footnotemark, and hence
$\barGamma\circ\bariota$ is invertible. \qed

\footnotetext{Hints:
For $\Lambda\circ(\barGamma\circ\bariota) = I_{\calB}$ note that the
stirring process of Figure~\ref{fig:stirring} can be carried out with the
green lines already pulled down as in Figure~\ref{fig:ReverseGamma} and
when looking from the side, one sees a dance of braid diagrams, which
is an equivalence of braids. For $(\barGamma\circ\bariota)\circ\Lambda
= I_{\calROU}$ one has to start from a whisk $W$ of the form of (C) of
Figure~\ref{fig:ReverseGamma} (namely, a whisk that when considered from
above, as in (A) of Figure~\ref{fig:ReverseGamma}, appears to be made of
$n$ straight vertical bars and $n$ non-intersecting planar strands). Then
one has to show that stirring tahini with parsley lines using $W$ will recreate
the shape of $W$ (minus the vertical bars) in the parsley lines.
}

Hence, we have constructed a separating braid invariant:

\begin{corollary} \label{cor:braids}
$\barGamma\circ\bariota$ is a complete invariant of braids. \qed
\end{corollary}

And in fact, in classical case, OU tangles are merely braids (though we will see in
Sections~\ref{sec:virtual} and~\ref{sec:more} that there is more to our story):

\begin{corollary} \label{cor:cOUisBraids}
All OU tangles are equivalent to braids. \qed
\end{corollary}

\begin{corollary} The two actions of Corollary~\ref{cor:actions} of
braids on reduced OU diagrams are simple and transitive.\qed
\end{corollary}

\draftcut \section{The Virtual Case} \label{sec:virtual}

Much is already written about virtual knot theory (see for example
\cite{Kauffman:VirtualKnotTheory, Kauffman:RotationalVirtualKnots,
Manturov:VirtualKnots}) here we give a quick summary of some basic ideas. Classical knots, braids, and tangles can all be
defined following the mould ``properly annotated planar graphs with
univalent and quadrivalent vertices and with properties PPP, modulo local
relations RRR''. Virtual knots, braids, and tangles are exactly the same,
except that the word ``planar'' is removed from the mould and otherwise
nothing is changed. See some examples in Figure~\ref{fig:vexamples}.

\begin{figure}
\[ 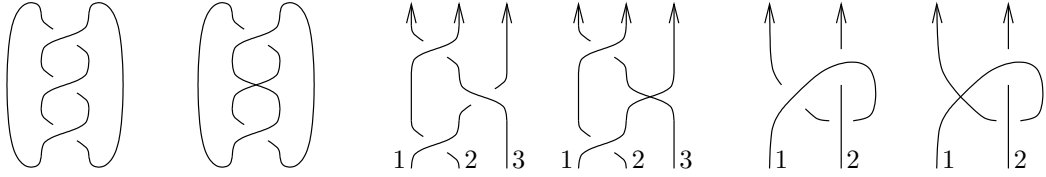 \]
\caption{
  A 3-crossing knot and a 2-crossing virtual knot, a 3-crossing braid and
  a 2-crossing virtual braid, and a 3-crossing tangle and a 2-crossing
  virtual tangle.
} \label{fig:vexamples}
\end{figure}

Note that all the virtual examples in Figure~\ref{fig:vexamples} contain
a feature like $\virtualcrossing$, often called a ``virtual crossing''. A
``virtual crossing'' is {\bf not a crossing}: It is merely an artifact of
the fact that when a non-planar graph is drawn on a piece of paper, some
edges will intersect, even though from a graph-theoretic
perspective these intersections are not vertices, and not part of the data of the graph.

In this paper virtual tangles and virtual braids are
always ``pure'': the ordering of the ends of strands around the boundary
of a planar domain has no graph theoretical meaning, for the planar
domain itself has no graph theoretic meaning. Yet it makes sense to
consider virtual objects whose strands are labelled by some finite set
$S$, and once this is done, virtual tangles become a monoid and virtual
braids become a group, where the product\footnote{ 
In the ``Geography vs.\ Identity'' language of~\cite{Talk:GvI},
compositions of classical tangles/braids are ``Geography'', because
they are defined using the placements of the ends being stitched, while
compositions of virtual tangles/braids are ``Identity'' because they
are defined using the identity of the ends being stitched.} of $T_1$ (or $B_1$) with $T_1$
(or $B_2$) is the disjoint union operation of graphs, followed by the
``stitching'' of the head of strand $a$ in $T_1$ (or $B_1$) to the tail of
strand $a$ in $T_2$ (or $B_2$), for every $a\in S$.

Thus the virtual (pure) braid group on $n$ strands is the group with
generators $\sigma_{ij}$ ``strand $i$ crosses over strand $j$ in a
positive crossing'' where $i\neq j\in\underline{n}$ and $\underline{n}$ is
some fixed set with $n$ elements (perhaps $\underline{n}=\{1,\ldots,n\}$),
and with relations matching the R3 move and the fact that crossings that
involve totally distinct strands commute:
\[ \vcalPB_n = \left\langle
    \sigma_{ij}\colon\ 
    \sigma_{ij}\sigma_{ik}\sigma_{jk}=\sigma_{jk}\sigma_{ik}\sigma_{ij}
    \text{ and }
    \sigma_{ij}\sigma_{kl}=\sigma_{kl}\sigma_{ij}
  \right\rangle,
\]
where it is understood that $i,j,k,l$ are arbitrary distinct elements of
$\underline{n}$. For example, the two braids in Figure~\ref{fig:vexamples}
are $\sigma_{12}\sigma^{-1}_{31}\sigma_{23}$ and $\sigma_{12}\sigma_{23}$
(the first was introduced as a classical braid, but it is also a pure
virtual braid).

We let $\vcalPBD_n$ denote the monoid of all virtual braid diagrams on $n$ strands: namely, the monoid of all
words in the generators $\sigma_{ij}^\pm$, with no relations.

With this said, everything in Section~\ref{sec:classical} up to but
not including the Classical Isomorphism Theorem (\ref{thm:iso}) makes sense and holds true in the
virtual case as well, for nothing there depends on the planarity of
diagrams. Hence we have a commutative diagram:
\[ \xymatrix{
  \vcalPBD_n \ar@{^{(}->}[r]^{\iota_v} \ar[d] & \vcalACD_n \ar[r]^{\Gamma_v} \ar[d] & \vcalOUD_n \ar[d] \\
  \vcalPB_n \ar[r]^<>(0.5){\bariota_v} \ar@/_6mm/[rr]_\Ch
    & \vcalAC_n \ar[r]^<>(0.5){\barGamma_v}_<>(0.5)\cong & \vcalROU_n
} \]
In this diagram everything was already defined or is the obvious virtual
analog of its counterpart in the classical case and does not need
a definition, except that we give a special name, the {\em Chterental
map} $\Ch\coloneqq\barGamma_v\circ\bariota_v$, to the composition along the bottom. Yet in contrast with
the Classical Isomorphism Theorem (\ref{thm:iso}) we have the Theorem~\ref{thm:vinj} below. This theorem is originally due to Oleg
Chterental~\cite{Chterental:VBandVCD, Chterental:Thesis} in a different context; our version is reformulated and independently proven, see a comparison in Discussion~\ref{disc:Chterental}. The proof presented in this section also leads to a new graph-valued braid and virtual braid invariant, see Discussion~\ref{disc:EG} and Section~\ref{sec:comp}.

\begin{theorem} \label{thm:vinj} (\cite{Chterental:VBandVCD, Chterental:Thesis}, independent proof
below)
$\Ch=\barGamma_v\circ\bariota_v$, and hence $\bariota_v$, is injective but not surjective.
\end{theorem}

Hence the following corollaries hold true:

\begin{corollary} \label{cor:vbraids} ~\cite{Chterental:VBandVCD, 
Chterental:Thesis}
$\Ch$ is a
complete invariant of virtual pure braids.\footnote{An earlier separation
result for virtual braids is in~\cite{GodelleParis:WordProblems}.} \qed
\end{corollary}

\begin{corollary} ~\cite{Chterental:VBandVCD, 
Chterental:Thesis} The two actions of virtual pure braids on reduced virtual
OU diagrams are simple but not transitive.\qed
\end{corollary}

\begin{corollary} ~\cite{Chterental:VBandVCD, 
Chterental:Thesis} Not all virtual OU tangles are equivalent to virtual pure braids. \qed
\end{corollary}

\begin{discussion} \label{disc:plan}
The rest of this section is devoted to a proof of
Chterental's Theorem (\ref{thm:vinj}). The idea is to ``extract'' as much of a virtual
braid out of a virtual OU tangle $T$ as possible, by extracting one
braid generator at a time while reducing the complexity of what remains
of $T$. The process won't always invert $\Ch$
(for $\Ch$ is not invertible), yet it will
invert $\Ch$ on the image of virtual braids,
which is enough. The main tools will be the Division Lemma (\ref{lem:divquo}) which
gives a necessary and sufficient condition for the extraction of one
braid generator, and the Diamond Lemma (\ref{lem:diamond}), which
will guarantee that this extraction process always terminates with a
well-defined answer.
\end{discussion}

\begin{definition} If $T\in\vcalAC_n$ is a virtual acyclic tangle,
let $\xi(T)$ denote the crossing number of $\barGamma_v(T)$, its R1-
and R2-reduced OU form (not counting virtual crossings, of course). We say
that a virtual braid $\beta\in\vcalPB_n$ divides a virtual acyclic
tangle $T\in\vcalAC_n$, and write $\beta\mid T$, if when $\beta$ is
extracted out of $T$, this reduces the crossing number. In other words,
if $\xi(\beta^{-1}T)<\xi(T)$. In that case, we call $\beta^{-1}T$ the
quotient of $T$ by $\beta$.
\end{definition}

\parpic[r]{
  \def\c{$\sigma_{12}$}
  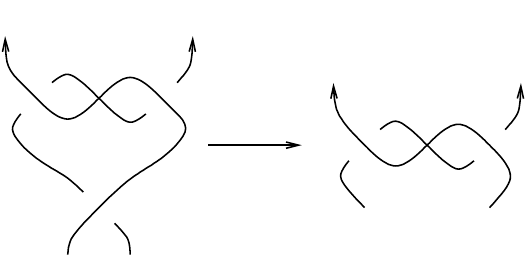
}
\begin{example} \label{exa:indivisible} The figure on the right shows two virtual OU
tangles, $T_1$ and $T_2$. We have that $\sigma_{12}\mid T_1$ and
$\sigma_{12}^{-1}T_1=T_2$. On the other hand, $T_2$ is not divisible
by anything, as it can be readily verified that $\xi(T_2)=2$
while $\xi(\sigma_{12}^{\pm 1}T_2)>2$ and $\xi(\sigma_{21}^{\pm 1}T_2)>2$.
\end{example}

\Needspace{30mm} 
\parpic[r]{
  \def\a{$\sigma_{23}$} \def\b{$\sigma_{13}$} \def\c{$\sigma_{12}$}
  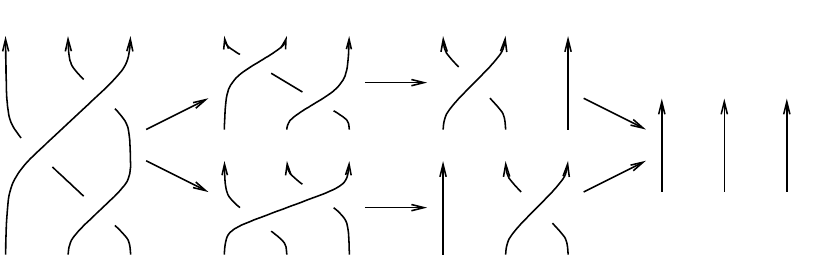
}
\begin{example} \label{exa:GarsideHexagon} The figure on the right shows in its left part
the Garside ``positive half twist'' braid on 3 strands, which
happens to be OU in its given presentation, fit within a hexagon
summarizing its five divisors $\sigma_{12}$, $\sigma_{23}$,
$\sigma_{12}\sigma_{13}$, $\sigma_{23}\sigma_{13}$,
$\sigma_{12}\sigma_{13}\sigma_{23} = \sigma_{23}\sigma_{13}\sigma_{12}$,
and the five resulting quotients. This hexagon is also an example of an
extraction graph; see Discussion~\ref{disc:EG}.
\end{example}

Please bear with us and read the following two examples carefully,
as they play a role in the proof of Chterental's Theorem (\ref{thm:vinj}).

\begin{example} \label{exa:CinnamonRolls} The $k$-twist braids are the braids
$(\sigma_{12}\sigma_{21})^{k/2}$ (for even $k$) or
$\sigma_{21}(\sigma_{12}\sigma_{21})^{(k-1)/2}$ (for odd $k$). They
are shown along with their reduced OU forms, the Cinnamon Roll
tangles $\CR_k$, in Figure~\ref{fig:CinnamonRolls}. Clearly,
$\sigma_{21}\mid \CR_{2k+1}$ with $\sigma_{21}^{-1}\CR_{2k+1}=\CR_{2k}$
and $\sigma_{12}\mid \CR_{2k}$ with $\sigma_{12}^{-1}\CR_{2k}=\CR_{2k-1}$,
and so we have the following chain of divisibilities and quotients:
\begin{equation} \label{eq:CinnamonRolls}
   \xymatrix{
    \cdots \ar[r] &
    \CR_4 \ar[r]^{\sigma_{12}} &
    \CR_3 \ar[r]^{\sigma_{21}} &
    \CR_2 \ar[r]^{\sigma_{12}} &
    \CR_1 \ar[r]^{\sigma_{21}} &
    \CR_0.
  }
\end{equation}
\end{example}

\begin{figure}
\[
  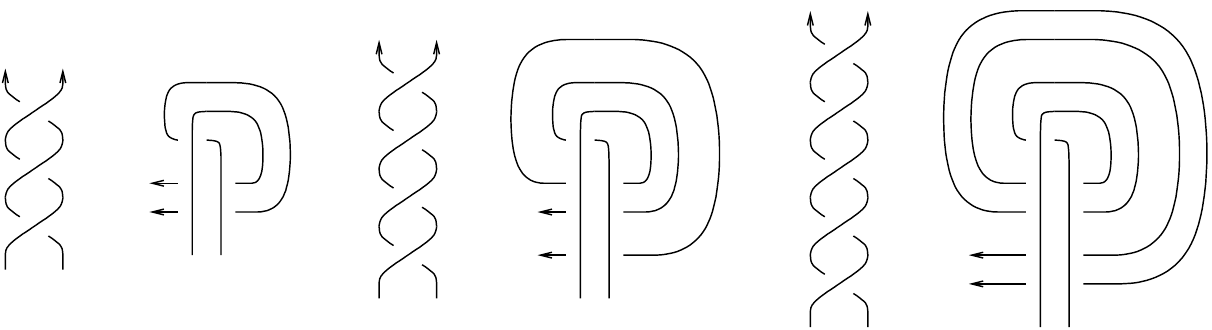\qquad
  \raisebox{8mm}{\includegraphics[height=0.75in]{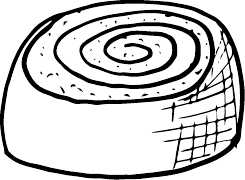}}
\]
\caption{
  The 3-twist, 4-twist, and 5-twist braids, and their reduced OU forms the
  Cinnamon Roll tangles $\CR_3$, $\CR_4$, and $\CR_5$. The equivalence
  of the twist braids with their respective cinnamon rolls should
  be clear to anyone who has observed how a kink in a band becomes a
  twisted band upon tugging. The bonus cinnamon roll was purchased from
  \url{https://thenounproject.com/}.
} \label{fig:CinnamonRolls}
\end{figure}

\begin{example} \label{exa:SCR} A slashed cinnamon roll is a cinnamon roll with an
extra always-over strand separating the O-parts of its two curving strands, as shown
in (A) of Figure~\ref{fig:SCR}. A slashed cinnamon roll
is divisible by both $\sigma_{21}^{-1}$ and $\sigma_{23}$, and
the quotients, after reductions by many R2 moves, are (B) and (C) of
Figure~\ref{fig:SCR}. These quotients are themselves cinnamon rolls
(with extras on the side), and so they can be divided and reduced further as
in Example~\ref{exa:CinnamonRolls}, leading to (D) and (E)
of Figure~\ref{fig:SCR}. Note also that (D)
can be reduced to (E) by dividing
first by $\sigma_{23}$ and then by $\sigma_{12}$, as shown.
Finally, note that we have two paths going from (A) to (E), via (B)
and (D) and via (C), and that each defines a braid word by reading the
divisors along it. We claim that these two braid words are equal in
$\vcalPB_3$. Namely, that
\begin{equation} \label{eq:SCRBraids1}
   \sigma_{21}^{-1}\sigma_{13}\sigma_{31}\sigma_{13}\sigma_{31}\sigma_{13}\sigma_{23}\sigma_{21}
  = \sigma_{23}\sigma_{13}\sigma_{31}\sigma_{13}\sigma_{31}\sigma_{13}.
\end{equation}
Indeed, to see the equality, slide strand 2 across the 5-twist
in the following picture\footnote{The equality also follows from
Chterental's Theorem (\ref{thm:vinj}), but we haven't proven Chterental's Theorem
yet, and in fact, the proof of Chterental's Theorem depends on the equality.}:
\[ 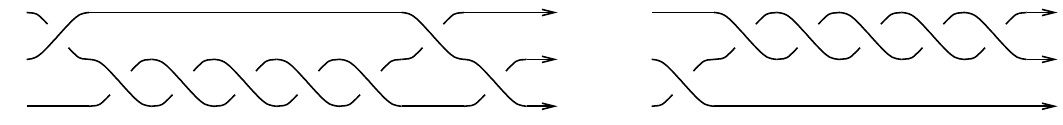. \]
\end{example}

\begin{figure}
\[ \def\s#1{$\sigma_{#1}$} \def\sb#1{$\sigma_{#1}^{-1}$}
  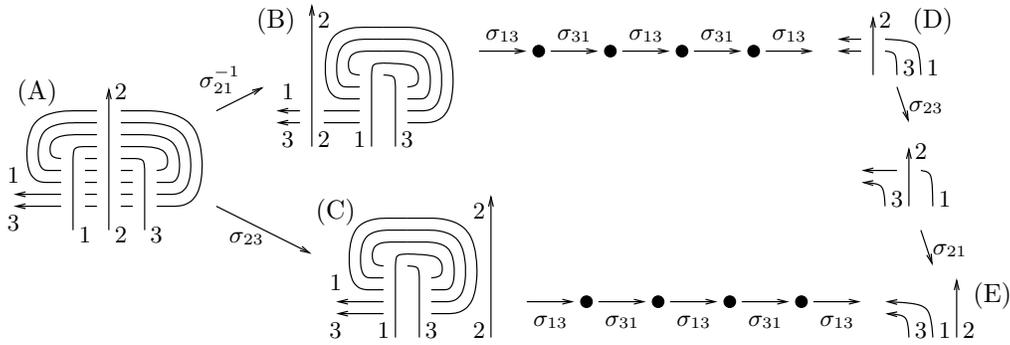
\]
\caption{A slashed cinnamon roll and its quotients down to the identity.} \label{fig:SCR}
\end{figure}

\begin{discussion} \label{disc:divquo} Next, we would like to understand
precisely when does a braid generator $\sigma_{ij}$ divide a reduced virtual
OU tangle $T$, and what is the quotient $\sigma_{ij}^{-1}T$, as a
reduced OU tangle. This is done in Figure~\ref{fig:divquo}. In (A)
of that figure we display $\sigma_{ij}^{-1}$ at the bottom and $T$ at
the top. $T$ could be complicated, but it turns out we only care about
what it looks like near the O part of strand $j$\footnote{``Near'' in
a combinatorial sense, meaning ``one or two crossings away from''. Not
in any metric sense, of course.}. So in (A) we also display strand $i$
just to remember that it exists, and the O part of strand $j$, up to its
transition point the $\bowtie$. In that part strand $j$ crosses over
a number of other strands, or over its own U part, or $i$'s U part,
and perhaps with multiplicity. We summarize that by showing only two
strands passing under, with no care for their identity or orientation.

\begin{figure}
\[
  \def\i{$i$} \def\j{$j$} \def\g{$\gamma$} \def\vp{$\vdots p$} \def\s{$\longrightarrow\atop\sigma_{ij}$}
  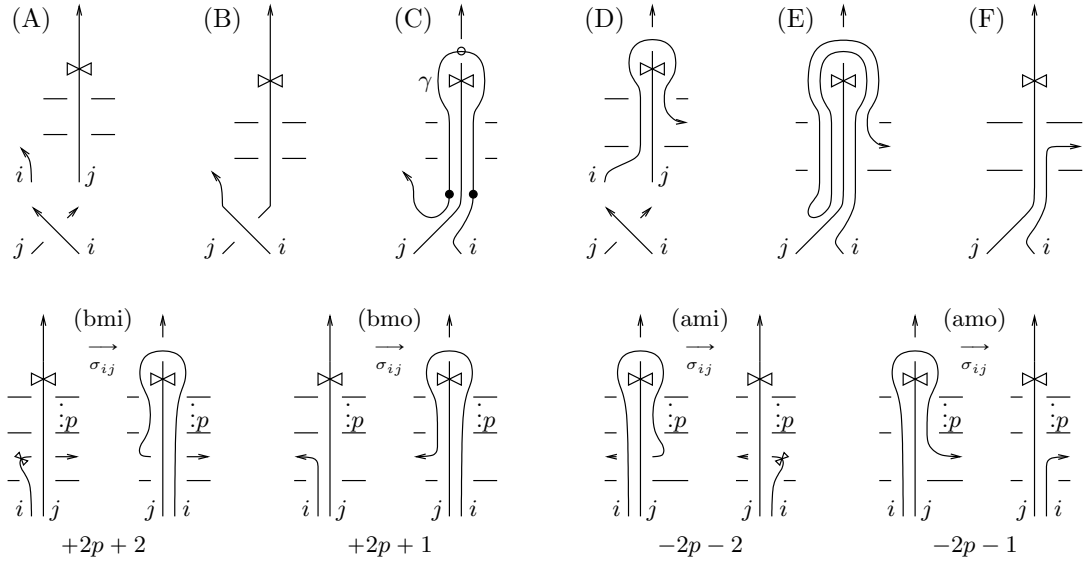
\]
\caption{Everything we need to know about divisibility and quotients.}
\label{fig:divquo}
\end{figure}

In (B) of Figure~\ref{fig:divquo} we attach $\sigma_{ij}^{-1}T$ to
$T$. The result is typically not OU and not reduced. In (C) we glide
the part where $i$ goes over $j$ past the $\bowtie$, to make the result
OU. We indicate the part of strand $i$ that got moved,
from one $\bullet$ to the other, by $\gamma$ and note that $\gamma$ has a natural
mid-point, indicated with a $\circ$. Note that the tangle in (C) might
not be reduced! That would be the case if as in (D), strand $i$ was to
follow $\gamma$ (backwards) at least a part of the way. For had this
been the case, the OU form of $\sigma_{ij}^{-1}T$ would look like in
(E), and would be reducible to (F) by R2 moves.

Ergo we care to know precisely how far backwards along $\gamma$ strand
$i$ follows, and when it deviates, precisely how. Four options for the
behaviour of $i$ are shown in the lower half of Figure~\ref{fig:divquo}:
The option ``before middle inward'', (bmi), means that $i$ traces along
$\gamma$ to before its mid-point, and then deviates by reaching its own
transition point $\bowtie$ and turning inwards, to cross under $j$. In
Figure~\ref{fig:divquo} we show both $T$ and the reduced OU form of
$\sigma_{ij}^{-1}T$ for option (bmi). If in $T$ there are $p$ further
strands passing under $j$ after the deviation point, it is easy to see
that $\sigma_{ij}^{-1}T$ gains $2p+2>0$ crossings over $T$. This is
indicated at the bottom of the (bmi) part of Figure~\ref{fig:divquo}.

The remaining three options for $i$ are shown in Figure~\ref{fig:divquo}
following the same pattern. In option (bmo) strand $i$ deviates from
$\gamma$ before the mid-point and turns outwards, into parts of $T$
we don't display. Here again $\xi(\sigma_{ij}^{-1}T)>\xi(T)$, with a
gain of $2p+1$. In option (ami) strand $i$ follows $\gamma$ past the
mid-point and turns inwards, and in (amo) it turns outwards. In the
last two cases $\sigma_{ij}^{-1}T$ loses crossings relative to $T$,
with the precise losses as indicated.

We leave it to the reader to verify that the four options (bmi), (bmo),
(ami), and (amo) are mutually exclusive and complete, and that in all
cases, if $T$ is reduced to start with, then $\sigma_{ij}^{-1}T$ as shown
in Figure~\ref{fig:divquo} is again reduced\footnotemark.
\footnotetext{In short, if $T$ is reduced and $T'$ is obtained from
it by adding and/or removing a number of crossings, when is $T'$
non-reduced? If an R1 or an R2 got added, or if a crossing got added
which along with an existing crossing creates an R2, or if crossings
are removed between a pair of existing or newly added crossings so as
to remove the separation between them and turn them into an R2 pair,
or if crossings are removed along a kink to create an R1. One must inspect
that none of these possibilities can occur here.}

Finally we note that we could repeat the whole discussion
for $\sigma_{ij}T$, and everything would be the same,
with only a left-right reflection of all the tangles in
Figure~\ref{fig:divquo}. \endpar{\ref{disc:divquo}}
\end{discussion}

Discussion~\ref{disc:divquo} proves the following lemma, which summarizes it:

\begin{lemma}[Division] \label{lem:divquo} Let $g=\sigma_{ij}^{\pm 1}$ be a generator of $\vcalPB_n$ and $T$
be a reduced virtual OU tangle.
\begin{enumerate}[leftmargin=*,labelindent=0pt]
\item $\xi(gT)$ is never equal to $\xi(T)$, so
  always, either $g^{-1}\mid T$ or $g\mid gT$.
\item $\sigma_{ij}|T$ if and only if $i$ is parallel\,\footnotemark\ to $j$
  on its left to its transition point $\bowtie$, and then immediately
  crosses over $j$ in a positive crossing, as in (ami) and (amo) of
  Figure~\ref{fig:divquo}. Similarly for $\sigma_{ij}^{-1}\mid T$, with
  ``left'' replaced with ``right'' and ``positive'' with ``negative''.
\item If indeed $g\mid T$, the quotient
  $g^{-1}T$ is determined by how far $i$ pushes backwards
  on the other side of $j$, past the point where it crosses over $j$,
  and by whether it turns ``in'' or ``out'' after that. \qed
\end{enumerate} \end{lemma}
\footnotetext{\label{foot:parallel}%
Note that we are in topology /
combinatorics, not in geometry, so ``$i$ is left-parallel to $j$''
means ``anything $j$ does $i$ does in tandem'', and not ``$i$ and $j$
maintain a constant distance between them''. More precisely, ``$i$
is left-parallel to $j$'' means ``any strand that crosses under $j$
in a positive crossing then crosses under $i$ in a positive crossing
(with no other crossings in between), any strand that crosses under $j$
in a negative crossings crossed under $i$ right before in a negative crossing
(with no other crossings in between), and $i$ and $j$ encounter those
pairs of crossings in the same order''.}

For the continuation of the proof of Chterental's Theorem (\ref{thm:vinj}) we
will need the Diamond Lemma. For completeness we provide a full
formulation and a proof.  While it is well-known~\cite{Bergman:Diamond,
Sapir:CombinatorialAlgebra, Smolka:Confluence}, we were not able to find
a simple exposition in a language sufficiently similar to ours.

\begin{definition} \label{def:diamond}
A binary relation $\to$ defined on a set $\calX$
is called Noetherian if there are no infinite sequences $(x_i)\in\calX$
such that $x_1\to x_2\to\ldots$ (in particular, never $x\to x$,
for $x\in\calX$). The ``transitive closure'' of $\to$, denoted
$\toto$, is the binary relation on $\calX$ defined by
\[ (x\toto y)
  \quad\Longleftrightarrow\quad
  (\exists x_0,\ldots,x_n\in\calX\text{ such that }x=x_0\to x_1\to\ldots\to x_n=y).
\]

\Needspace{17mm} 
\parpic[r]{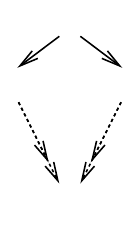}
\noindent It is clear that $\toto$ is transitive and taking $n=0$ we see that it
is reflexive. A non-empty subset $\calY$ of $\calX$ is called connected
if whenever $y\in \calY$ and $x\in\calX$ satisfies $x\to y$ or $y\to x$,
then $x\in \calY$. We say that $\to$ satisfies the diamond condition
if for every $a,b,c\in\calX$ such that $a\to b$ and $a\to c$, there is
some $d\in\calX$ such that $b\toto d$ and $c\toto d$ (``every wedge can be completed to a diamond'', as on the right).
\end{definition}

\begin{lemma} \label{lem:diamond} (The Diamond Lemma,
\cite{Newman:DiamondLemma}) If a Noetherian relation $\to$ on a set
$\calX$ satisfies the diamond condition then every connected subset
$\calY\subset\calX$ has a unique final element. Namely, there is a
unique $f\in \calY$ such that for every $y\in \calY$, $y\toto f$.
\end{lemma}

\noindent{\it Proof.} If $t\in\calZ\subset\calX$, we say that $t$ is
$\calZ$-terminal if there is no $z\in\calZ$ with $t\to z$. By
the Noetherian property, every non-empty $\calZ$ has a terminal element
(perhaps many). Set
\[ \calG\coloneqq
  \{x\in \calX\colon\text{there is a {\em unique} $\calX$-terminal $\tau(x)$ such that $x\toto \tau(x)$}\}.
\]
Clearly if $x\in\calG$ and $x\toto y$, then $y\in\calG$ and
$\tau(x)=\tau(y)$ (*). If $\calB\coloneqq\calX\setminus\calG$ is
non-empty, pick some $\calB$-terminal element $a\in\calB$. If $b,c\in\calX$ and $a\to b$
and $a\to c$, find $d$ such that $b\toto d$ and $c\toto d$. As $a$ is
$\calB$-terminal, $b,c,d\in\calG$ so by (*) $\tau(b)=\tau(d)=\tau(c)$.
Hence all the followers $b$ of $a$ have the same $\tau(b)$, and hence
$a\in\calG$ with $\tau(a)=\tau(\text{any follower})$ (if $a$ has no
followers take $\tau(a)=a$). But this contradicts $a\in\calB$, so $\calB$
is empty and $\calG=\calX$.

Now if $x,y\in\calX$ and $x\to y$ then $\tau(x)=\tau(y)$, so by connectivity
$\tau$ is constant on $\calY$. Call that constant $f$. \qed

\begin{definition} Let $\calX_n=\vcalPB_n\times\vcalROU_n$. We define a binary relation $\to$ on $\calX_n$ as follows
\begin{eqnarray*}
  (\beta_1,T_1)\to(\beta_2,T_2)
  & \Longleftrightarrow &
  \text{for some }g=\sigma_{ij}^{\pm 1}:\ g\mid T_1,\ T_2=g^{-1}T_1, \text{ and }\beta_2=\beta_1g \\
  & \Longleftrightarrow &
  \text{for some }g=\sigma_{ij}^{\pm 1}:\ \xi(T_2)<\xi(T_1),\  T_2=g^{-1}T_1, \text{ and }\beta_2=\beta_1g.
\end{eqnarray*}
\end{definition}

\begin{example} With a bit of thought, four examples of elements (A),
(B), (C), and (D) of $\calX_3$, in fact of $\calPB_3\times\calROU_3$, can
be seen in Figure~\ref{fig:stirring}. Precisely, the ``whisk'' part of
each of the figures is the braid part $\beta$, and the ``parsley'' part
becomes an OU tangle if the whisk is replaced by a straight ``identity''
whisk as in image (D). These elements are related in the opposite manner
to the figure: (D)$\to$(C)$\to$(B)$\to$(A).
\end{example}

\begin{discussion} \label{disc:to}
Note that if $(\beta_1,T_1)\to(\beta_2,T_2)$ then $\beta_1T_1=\beta_2T_2$
and $T_2$ is ``simpler'' than $T_1$. Thus ``flowing with $\to$''
agrees with our plan from Discussion~\ref{disc:plan}.

Note also that if $(\beta_1,T_1)\to(\beta_2,T_2)$ then $\beta_2$
and $T_2$ are determined by $\beta_1$ and $T_1$ and a single
generator $g$ of $\vcalPB_n$, which we can mark atop the $\to$ symbol as
$(\beta_1,T_1)\overset{g}{\longrightarrow}(\beta_2,T_2)$.  With this in
mind, a $\toto$-relation in $\calX_n$, meaning a $\to$-chain, is determined by a pair
\[ \left(
  \beta_0,
  \xymatrix{T_0 \ar[r]^{g_0} & T_1 \ar[r]^{g_1} & T_2 \ar[r]^{g_2} & \cdots \ar[r]^{g_{m-1}} & T_m}
\right) \]
where $\beta_0$ is a virtual braid, and where each $g_k$ is a generator of $\vcalPB_n$ and for every $k$,
$g_k\mid T_k$ and $T_{k+1}=g_k^{-1}T_k$. The $\to$-chain corresponding to such a pair is
\[ \xymatrix{
  (\beta_0,T_0) \ar[r]^<>(0.5){g_0} &
  (\beta_0g_0,T_1) \ar[r]^<>(0.5){g_1} &
  (\beta_0g_0g_1,T_2) \ar[r]^<>(0.5){g_2} &
  \cdots \ar[r]^<>(0.5){g_{m-1}}
  & (\beta_0\prod g_k,T_m)
} \]
An example with $\beta_0$ suppressed is in Example~\ref{exa:CinnamonRolls}. It can be completed by choosing
$\beta_0$ arbitrarily.

Finally, note that a diamond in $\calX_n$ is determined a single virtual
braid $\beta_0$ and two chains as above with a shared initial tangle,
\begin{equation} \label{eq:DiamondInX}
  \xymatrix@R=0mm@C=14mm{
    & T_1 \ar[r]^{g_1} & T_2 \ar[r]^{g_2} & \cdots \ar[r]^{g_{m-2}} & T_{m-1} \ar[rd]^<>(0.5){g_{m-1}} & \\
    T_0=T'_0 \ar[ur]^<>(0.5){g_0} \ar[dr]^<>(0.5){g'_0} & & & & & T_m=T'_{m'}, \\
    & T'_1 \ar[r]^{g'_1} & T'_2 \ar[r]^{g'_2} & \cdots \ar[r]^{g'_{m'-2}} & T'_{m'-1} \ar[ur]^<>(0.5){g'_{m'-1}} &
  }
\end{equation}
with the additional requirement that $\prod g_k=\prod g'_k$ in $\vcalPB_n$
(which also implies that the chains share their end tangles). An example
of such a diamond, with the initial $\beta_0$ suppressed, is in
Figure~\ref{fig:SCR}. \endpar{\ref{disc:to}}
\end{discussion}

\begin{lemma} The relation $\to$ satisfies the conditions of the Diamond Lemma.
\end{lemma}

\noindent{\it Proof.} As crossing numbers are always finite and $\to$
decreases the crossing number of the tangle part, $\to$ is Noetherian. To
verify the diamond condition we must start with a reduced OU tangle $T_0=T$
and two generators $g_0$ and $g'_0$ that divide it, and ``complete
a diamond'' as in Equation~\eqref{eq:DiamondInX}. Let us start with the
hardest case.

\Needspace{25mm} 
\parpic[r]{\def\i{$i$} \def\j{$j$} \def\k{$k$}
  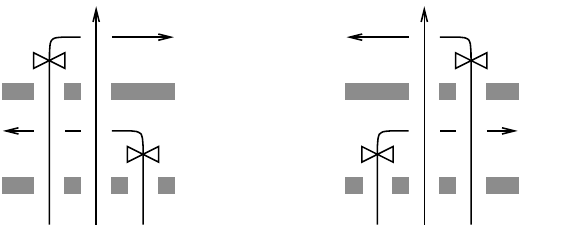
}
\noindent{\it Case 1. For some $i,j,k$, $\sigma_{ji}^{-1}\mid T$ and
$\sigma_{jk}\mid T$.} By (2) of the Division Lemma (\ref{lem:divquo}),
strand $j$ must be left-parallel to strand $k$ to $k$'s $\bowtie$ and
right-parallel to strand $i$ to $i$'s $\bowtie$. So the tangle $T$ must
contain a part as on the right, with the two grey bands representing any
number $w_1$ and $w_2$ of further strands. (Two options for $T$ are shown
and we will treat only the first, as the second is mirror image thereof).

In order to complete a diamond, we need to know the quotients
$\sigma_{ji}T$ and $\sigma_{jk}^{-1}T$. By (3) of the Division Lemma
(\ref{lem:divquo}), this gets complicated if $j$ continues as a right
parallel of $k$ and as a left parallel of $i$.

\Needspace{18mm} 
\parpic[r]{\def\i{$i$} \def\j{$j$} \def\k{$k$}
  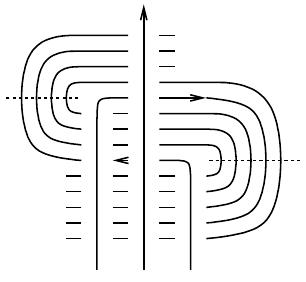
}
Sayeth you: ``That's impossible! One can't be to the left of $i$ and also to
the right of $k$!''. But unfortunately, it's the topological / combinatorial
``left'' and ``right'' that concern us here and the impossible is
actually possible. As indicated in Footnote~\ref{foot:parallel}, ``$j$
is right-parallel to $k$'' (to a point) just means that some number
$p_k\geq 0$ of strands that cross under $k$ then proceed to cross under
$j$, in the same order. As in the figure on the right, this can be drawn
while keeping $j$ straight and making the $p_k$ under-strands curve into
semi-circles. Similarly, ``$j$ is left-parallel to $i$'' (to a point)
means, as in the figure, that some number $p_i\geq 0$ of arcs cross
under both $i$ and $j$ in the manner shown.

Unfortunately, there are many cases to check, depending on the relative
sizes of the widths $w_1$ and $w_2$, of the ``push-back numbers''
$p_i$ and $p_k$, and on whether, at the end, $j$ crosses under $i$,
or under $k$, or goes elsewhere, as in options (ami) and (amo) of
Figure~\ref{fig:divquo}. We will try to make it as painless as possible.

The ``base cases'' occur when
\begin{enumerate}
\item $w_2=0$, 
\item $|p_i-p_k|\leq 1$, 
\item  $w_1$ is as small as it can be given $p_i$, $p_k$, and $w_2$
  (meaning, $w_1=p_k$ assuming the first two conditions hold).
\item strand $j$ continues outward relative to both $i$ and $k$,
  as in option (amo) of Figure~\ref{fig:divquo}.
\end{enumerate}
There are then three possibilities: If $p_k=p_i+1$, we are looking
at a slashed cinnamon roll as in Figure~\ref{fig:SCR}, that figure
also shows how to complete the diamond, and the required braid
relation is Equation~\eqref{eq:SCRBraids1}.  The cases $p_k=p_i$
and $p_k=p_i-1$ correspond to slashed cinnamon rolls rolled slightly
differently and are shown as (A) and (B) of Figure~\ref{fig:Cases23},
along with the corresponding diamonds and braids relations. Note
that in all of these cases the length of the ``twist sequence''
$\sigma_{ik}\sigma_{ki}\sigma_{ik}\cdots$ is $p_k+1$, so these diamonds
can be be arbitrarily long.

\begin{figure}
\[ \def\i{$i$} \def\j{$j$} \def\k{$k$}
  \def\s#1{$\sigma_{#1}$} \def\sb#1{$\sigma_{#1}^{-1}$}
  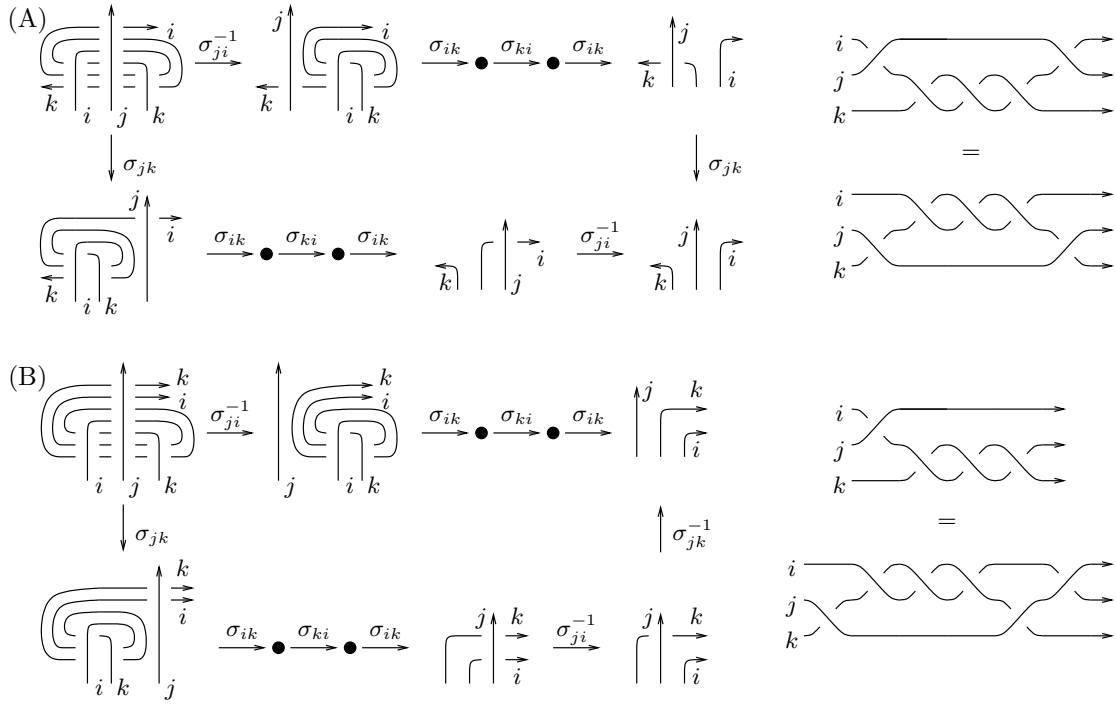
\]
\caption{Two other slashed cinnamon rolls.}
\label{fig:Cases23}
\end{figure}

What if $w_1$ is bigger than the least it can be given the other
parameters (which are otherwise unchanged)? That adds a band of strands at the bottom, as in (A) of
Figure~\ref{fig:WhatIfs}. This band gets added in the same way everywhere
else in Figures~\ref{fig:SCR} and~\ref{fig:Cases23}, with no change to
the resulting diamonds.

\begin{figure}
\[ \def\i{$i$} \def\j{$j$} \def\k{$k$}
  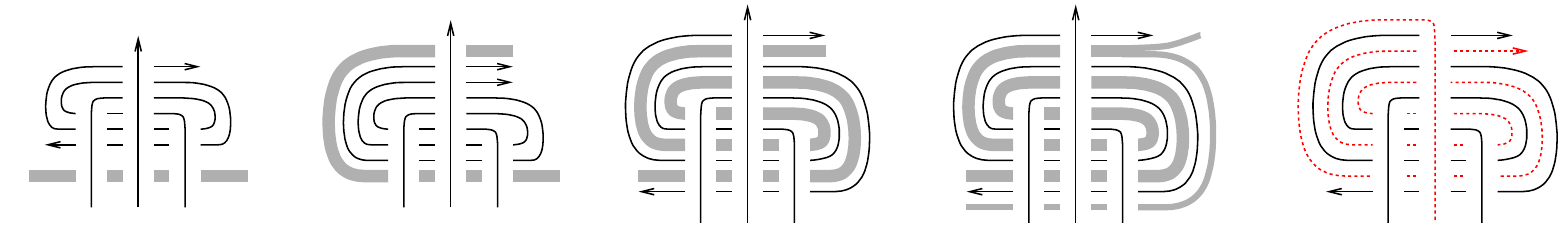
\]
\caption{The what ifs.}
\label{fig:WhatIfs}
\end{figure}

What if $p_i>p_k+1$ (yet respecting the other constraints)? This adds and extra band of strands as in (B)
of Figure~\ref{fig:WhatIfs}. These bands get tugged along through
the processes of Figure~\ref{fig:Cases23} with no changes to the end
results. A similar thing happens if $p_k>p_i+1$.

What if $w_2>0$ and $p_1=p_k$ are multiples of $w_2+1$? Then we are
in (C) of Figure~\ref{fig:WhatIfs}, and the slashed cinnamon roll has
a band of width $w_2$ of extra filling! One may check that the extra
filling unwinds along with the rest as in Figure~\ref{fig:Cases23}
with no change to the diamonds. There are similar ``filled'' versions
of the other base cases.

What if $w_2>0$ and $p_1=p_k$ are not round multiples of $w_2+1$? Then
we are in a situation like (D), which is a combination of previous cases,
and the same conclusions apply.

What if we are in an (ami) case instead of (amo)? We are in a situation
like in (E), and the same comments apply as for (C). We made $j$ dotted in
(E), to make the similarity with (C) easier to see.

What if several of the what ifs are combined? Then some combination of
(A)-(E) of Figure~\ref{fig:WhatIfs} applies, and we leave it to the reader
to verify that in all cases, diamonds complete as in Figures~\ref{fig:SCR}
and~\ref{fig:Cases23}.

\noindent{\it Case 2. For some $i,j,k$, $\sigma_{ij}\mid T$
and $\sigma_{jk}\mid T$.} By (2) of the Division Lemma (\ref{lem:divquo}),
strand $i$ must be a left-parallel of $j$ and then cross over it,
and $j$ must be a left-parallel of $k$ and then cross over it, as
shown in Figure~\ref{fig:CaseM}, along with the completion of the
wedge into a diamond. In that figure we took option (amo) for both
divisibilities. Option (ami) is possible only for the $\sigma_{ij}\mid T$
divisibility, and makes little difference to the resulting diamond.

\begin{figure}
\[ \def\i{$i$} \def\j{$j$} \def\k{$k$} \def\s#1{$\sigma_{#1}$}
  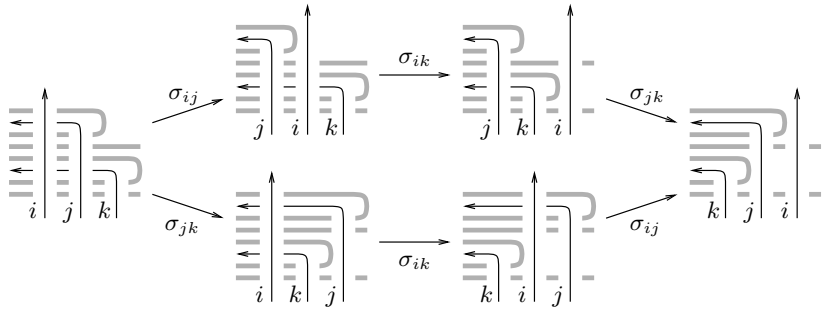
\]
\caption{Case 2 and the resulting diamond.} \label{fig:CaseM}
\end{figure}

\noindent{\it Case 3. For some $i,j,k$, $\sigma_{ij}^{-1}\mid T$
and $\sigma_{jk}^{-1}\mid T$.} That's the same as Case~2, with left interchanged with right.

\noindent{\it Case 4. For some distinct $i,j,k,l$, $\sigma_{ij}^{s_1}\mid
T$ and $\sigma_{kl}^{s_2}\mid T$, where $s_1,s_2\in\{\pm 1\}$.} In this
case division by $\sigma_{ij}^{s_1}$ commutes with division by $\sigma_{kl}^{s_2}$,
and the resulting diamond is a square, as in Figure~\ref{fig:Square}.

\begin{figure}
\[ \xymatrix@R=0mm@C=25mm{
  & \sigma_{ij}^{-s_1}T \ar[rd]^<>(0.5){\sigma_{kl}^{s_2}} & \\
  T \ar[ru]^<>(0.5){\sigma_{ij}^{s_1}} \ar[rd]_<>(0.5){\sigma_{kl}^{s_2}} & & \sigma_{kl}^{-s_2}\sigma_{ij}^{-s_1}T \\
  & \sigma_{kl}^{-s_2}T \ar[ru]_<>(0.5){\sigma_{ij}^{s_1}} &
} \]
\caption{The diamond for case 4.} \label{fig:Square}
\end{figure}
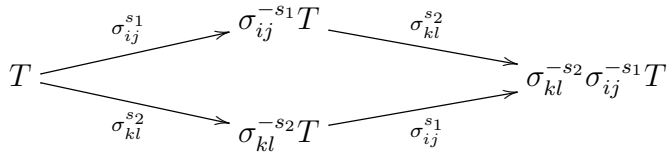

\noindent{\it There are no further cases to check.} If two generators
divide $T$, they involve at most 4 strands, and if they involve
exactly 4 strands, that's Case~4. The Division Lemma (\ref{lem:divquo}) excludes the
possibility that the two generators involve only two strands --- namely,
that they are two of $\{\sigma_{ij}^{\pm 1},\sigma_{ji}^{\pm 1}\}$. It
also excludes the remaining 3-strand cases: namely, that they are $\{\sigma_{ij},\sigma_{ik}\}$,
$\{\sigma_{ik},\sigma_{jk}\}$, $\{\sigma_{ij}^{-1},\sigma_{ik}^{-1}\}$, 
$\{\sigma_{ik}^{-1},\sigma_{jk}^{-1}\}$, $\{\sigma_{ij},\sigma_{jk}^{-1}\}$, $\{\sigma_{ij}^{-1},\sigma_{jk}\}$,
or $\{\sigma_{ij},\sigma_{kj}^{-1}\}$. Each of these exclusions requires a short argument, and we provide only the
argument for the first one. Indeed if both $\sigma_{ij}\mid T$ and $\sigma_{ik}\mid T$, then by
the Division Lemma strand $i$ is a
left parallel of both the O part of $j$ (call it $O_j$) and the O part
of $k$ (call it $O_k$), showing that at least one of $O_j$ and $O_k$ is
empty. Without loss of generality, it is $O_j$. But then the $i$ over $j$
crossing that the Division Lemma guarantees is the first crossing on
both $i$ and $j$, leaving no room for $O_k$ to be a right parallel of
a part of $i$, unless $O_k$ is also empty. But then the first crossing
on $i$ is both over $j$ and over $k$, which is impossible. \qed

\vskip 2mm
\noindent{\it Proof of Chterental's Theorem (\ref{thm:vinj}).} We have shown
that $\calX_n=\vcalPB_n\times\vcalROU_n$ with the relation $\to$
satisfies the conditions of the Diamond Lemma. Let
$f\colon\calX_n\to\calX_n$ be the function guaranteed by the Diamond
Lemma, mapping every element to the unique final element in its connected
component.

Let $I$ denote both the the 0-crossing pure virtual braid on $n$
strands (the identity element of $\vcalPB_n$) and the 0-crossing
OU tangle on $n$ strands. By (1) of the Division Lemma (\ref{lem:divquo}), if
$\beta',\beta''\in\vcalPB_n$ are virtual braids and $g=\sigma_{ij}^{\pm
1}$ is a generator of $\vcalPB_n$ then
\begin{align*}
  \text{either}\qquad &
  (\beta'g,\Ch(\beta'')) \to (\beta',g\Ch(\beta'')) = (\beta',\Ch(g\beta'')) &
  \qquad\text{if }g^{-1}\mid\Ch(\beta'') \\
  \text{or}\qquad &
  (\beta',\Ch(g\beta'')) = (\beta',g\Ch(\beta'')) \to (\beta'g,\Ch(\beta'')) &
  \qquad\text{if }g\mid g\Ch(\beta''),
\end{align*}
and so by induction on the length of a presentation of
$\beta\in\vcalPB_n$, $(\beta,I)$ and $(I,\Ch(\beta))$ are in the same
connected component of $\calX_n$. Hence $f(I,\Ch(\beta)) = f(\beta,I)
= (\beta,I)$, by the Diamond Lemma and as $\xi(I)=0$ implies that
$(\beta,I)$ is final.

Now if $\Ch(\beta_1)=\Ch(\beta_2)$ then
\[ (\beta_1,I) = f(\beta_1,I) = f(I,\Ch(\beta_1)) = f(I,\Ch(\beta_2)) = f(\beta_2,I) = (\beta_2,I), \]
so $\beta_1=\beta_2$, proving the injectivity of $\Ch$.

Note also that we learned that for every $\beta\in\vcalPB_n$,
$(I,\Ch(\beta))\toto(\beta,I)$, and in particular, $\Ch(\beta)$
must be divisible by at least one generator of $\vcalPB_n$. But
Example~\ref{exa:indivisible} exhibits a virtual OU tangle $T_2$ that
is not divisible by any generator, and hence $\Ch$ is not surjective. \qed

\draftcut \Needspace{16mm} 
\section{Assorted Comments} \label{sec:Assorted}

\begin{discussion} \label{disc:Chterental} Virtual OU tangles are
equivalent to Chterental's ``Virtual Curve Diagrams'' (VCDs)
\cite{Chterental:VBandVCD, Chterental:Thesis}, though we hope that they
are a bit more natural, and that they tell a bigger story. We explain
the relationship in Figure~\ref{fig:VCD}, albeit without repeating
Chterental's definitions. Given a virtual curve diagram as in (A) of
Figure~\ref{fig:VCD}, connect all the curve ends on the upper (dashed)
line to the vertical infinity using O curves (thus making everything
else into U curves), delete the upper and the lower lines, and get
a virtual OU tangle (B). It is positioned opposite to our habits\footnote{We
are in topology / combinatorics; these habits are anyway meaningless.}
so in order to feel a bit better, we flip the picture over in~(C).

\begin{figure}
\[ 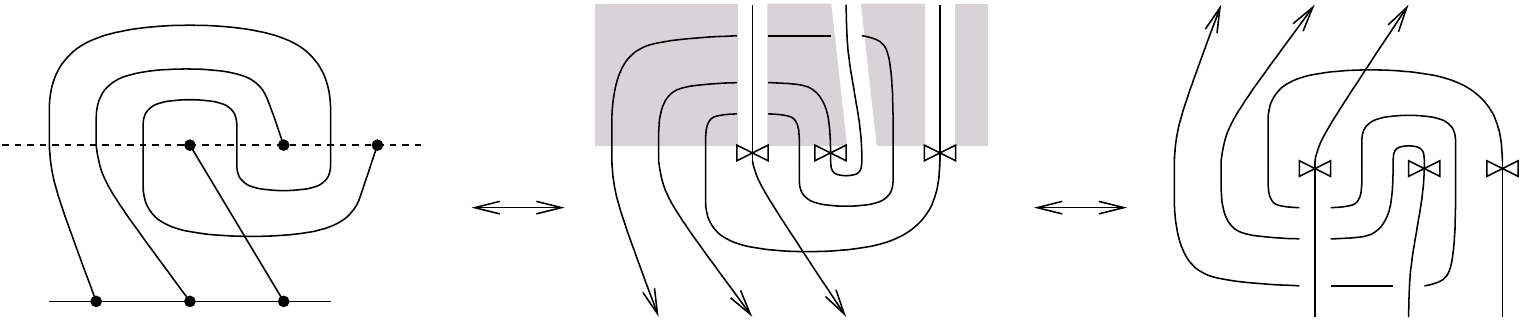 \]
\caption{A Chterental Virtual Curve Diagram (VCD) and the corresponding virtual OU tangle.} \label{fig:VCD}
\end{figure}

To go back, draw a virtual OU tangle with the O parts of its strands straight,
parallel, of equal length, and heading downward (that's always possible
as they never cross each other), and then draw the U parts curving between them,
perhaps with virtual crossings\footnote{We've emphasized that ``virtual
crossings'' are {\bf not crossings}. But here we must link with other
people's conventions.}. Push all the virtual crossings to below the areas
between the O strand-parts (in light grey in (B) of Figure~\ref{fig:VCD}),
re-insert an upper line and a lower line, and get back to (A) of
Figure~\ref{fig:VCD}, a VCD.

While in a different context, our proof of Chterental's Theorem (\ref{thm:vinj}) is similar in spirit to Chterental's proof that
VCDs can be used to separate virtual braids.  A notable difference
is that we fully analyze the possible diamonds, instead of relying on
the classical Artin's theorem. Another minor difference is that
we deal only with pure virtual braids (minor because separating braids
that induce different permutations is a non-issue). \endpar{\ref{disc:Chterental}}

\end{discussion}

\begin{remark} \label{rem:classical}
The Division Lemma (\ref{lem:divquo}) implies that if $T$ is classical
(namely, is given with a planar presentation in a disk $D$) and
$\sigma_{ij}^s\mid T$ with $s\in\{\pm 1\}$, then the beginning points
of strands $i$ and $j$ must be adjacent within the boundary of $D$
(with $i$ left of $j$ if $s=+1$ and $i$ right of $j$ if $s=-1$), and then
$\sigma_{ij}^{-s}T$ is classical again. By induction, if a virtual braid
$\beta$ divides a classical $T$, then $\beta$ is actually classical.
\end{remark}

\Needspace{28mm} 
\parpic[r]{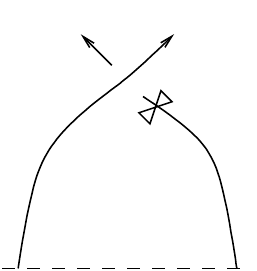}
\begin{remark} \label{rem:tent} Everything within the proof of
Chterental's Theorem (\ref{thm:vinj}) can be restricted to the classical case, hence
reproving (in a complicated and very algebraic manner) that the map
$\barGamma\circ\bariota$ of the Classical Isomorphism Theorem (\ref{thm:iso}) is injective. To show
by algebraic means that $\barGamma\circ\bariota$ is also surjective
it is enough to show that every non-trivial classical OU tangle $T$ is
divisible by at least one $\sigma_{ij}^{\pm 1}$ --- dividing and repeating
until the process terminates (it must, as crossing numbers decrease),
and using the previous remark would show that $T$ is equivalent to a
classical braid. The required ``existence of a divisor'' property is proven
as follows: For any strand $j$ let $i$ be the first strand to cross over
$j$ after $j$'s transition point $\bowtie$. If the starting points of
$i$ and $j$ are adjacent then either $\sigma_{ij}$ or $\sigma_{ij}^{-1}$
divides $T$.  Otherwise a ``triangular tent shield'' is created as on
the right, and the same argument can be repeated within it. When the
process terminates, we have a divisor. The topological arguments of
the classical braids Section (\ref{sec:classical}) are of course a lot simpler.
\end{remark}

\begin{remark} \label{rem:TwoSquares} Let $\calT_n$ denote the set of all classical tangles
with $n$ open strands and let $\vcalT_n$ denote the set of all virtual tangles
with $n$ open strands. We wish to briefly study the following two commutative
squares, the ``classical'' and the ``virtual'':
\[
  \xymatrix{
    \calB_n \ar[r]^\chi_{\text{1-1}} \ar[d]_\bariota^\cong & \calT_n \\
    \calAC_n \ar[r]^<>(0.5)\barGamma_<>(0.5)\cong & \calROU_n \ar[u]_\varphi^{\text{1-1}}
  }
  \qquad\qquad
  \xymatrix{
    \vcalPB_n \ar[r]^{\chi_v}_{\text{1-1?}} \ar[d]_{\bariota_v}^{\text{1-1}} & \vcalT_n \\
    \vcalAC_n \ar[r]^<>(0.5){\barGamma_v}_<>(0.5)\cong & \vcalROU_n \ar[u]_{\varphi_v}^{\text{1-1?}}
  }
\]
In these squares, $\bariota$, $\barGamma$, $\bariota_v$, and
$\barGamma_v$ along with the properties ($\cong$, $\cong$, 1-1,
and $\cong$) were discussed in Sections~\ref{sec:classical}
and~\ref{sec:virtual}. Also, $\chi$ ($\chi_v$) and $\varphi$
($\varphi_v$) are the obvious maps of (virtual) braids and reduced
(virtual) OU tangles into (virtual) tangles\footnote{In the vaguest
way, $\chi$ and $\varphi$ are pictograms for braids and OU tangles,
respectively.}. We note that the injectivity of $\chi$ was known already
to Artin~\cite[Theorem~12]{Artin:TheoryOfBraids}\footnote{Quick proof:
The fundamental group of the complement of a braid along with the $n$
bottom meridians and the $n$ top meridians determines the braid, and
this invariant extends to tangles.}, and thus it follows that $\varphi$
is also injective. We do not know if $\chi_v$ and $\varphi_v$ are injective.
The injectivity of $\chi_v$ was stated as an open problem in
\cite[Question~5.1]{AudoxBellingeriMeilhanWagner:UVWHomotopy}. Given the injectivity of $\bariota_v$, the
injectivity of $\chi_v$ would clearly follow from the
injectivity of $\varphi_v$, which we conjecture holds true. \endpar{\ref{rem:TwoSquares}}
\end{remark}

\begin{conjecture} The obvious map $\varphi_v$ of reduced virtual OU tangles into virtual tangles is injective.
\end{conjecture}

The reason we believe this conjecture is that we see a plausible path to
proving it. One way to go would be to find enough invariants of virtual
tangles to separate reduced virtual OU tangles. There are plenty of
invariants of virtual tangles coming from Hopf algebras and quantum
groups, reduced virtual OU tangles are easy to enumerate (they are
``free'' objects, subject to no relations), and there are precedents
where using quantum groups one can find enough invariants to separate
near-free objects: for example, quantum $gl(N)$ invariants separate
braids~\cite{Bar-Natan:glN}, and braid groups are semi-direct products
of free groups.

\begin{discussion} \label{disc:OUH} In fact, there is a very close
relationship between virtual OU tangles and Hopf algebras.  Denote by
$\vcalOU^p_q$ the set of OU tangles that have $p$ O-only strands
and $q$ U-only strands (it is a subset of $\vcalOU_{p+q}$). We claim
that $\vcalOU^p_q$ is precisely the set of ``universal formulas'' for
linear maps $\Hom(H^{\otimes p} \to H^{\otimes q})$, where $H$ is an
arbitrary involutive\footnote{Meaning that the antipode $S$ satisfies
$S^2=I$.} Hopf algebra\footnote{Or even, an involutive Hopf object in
a symmetric monoidal category.}: Meaning, those formulas that can be
written as an arbitrary composition of the structure maps $m$, $\Delta$,
$S$, $\epsilon$, and $\eta$ of $H$, and that make sense even if $H$
is infinite dimensional (so they contain no cycles).

\parpic[r]{\def\D{$\Delta$}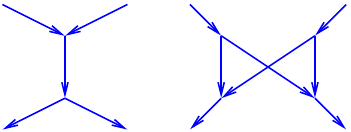}
By means of an example and with all details suppressed,
Figure~\ref{fig:HopfWords} demonstrates how a
virtual O/U tangle becomes a Gauss diagram and then a universal Hopf
formula. Furthermore, one may show that the relation between the product $m$ and the
coproduct $\Delta$ in a Hopf algebra (illustrated on the right) can be
used to bring all coproducts in a universal Hopf formula to before all
the products, and hence every universal Hopf formula comes from an O/U
tangle as in Figure~\ref{fig:HopfWords}. \endpar{\ref{disc:OUH}}

\end{discussion}

\begin{figure}
\[ \def\D{$\Delta$} 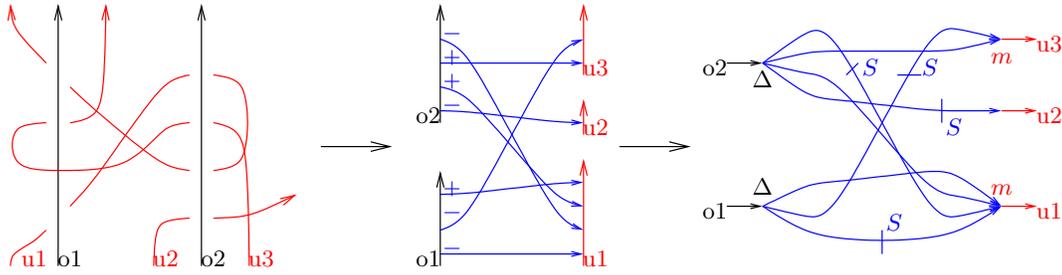 \]
\caption{
  A virtual O/U tangle in $\vcalOU^2_3$ becomes a Gauss diagram becomes
  a universal Hopf formula representing an element of $\Hom(H^{\otimes
  2} \to H^{\otimes 3})$. Note that the antipode $S$ is inserted on the
  $(-)$-marked edges of the Gauss diagram, which correspond to the negative
  crossings of the tangle.
} \label{fig:HopfWords}
\end{figure}

\begin{remark} The awkwardness of having to restrict to involutive Hopf algebras suggests that there may be an
alternative way to tell the story of this paper that does not require involutivity. Perhaps using ``rotational virtual
tangles''~\cite{Kauffman:RotationalVirtualKnots}.
\end{remark}

\begin{remark} The map $\Ch\colon\vcalPB_n\to\vcalROU_n$ along with
Discussion~\ref{disc:OUH} imply that there is an invariant of virtual
braids with values in $\End(H^{\otimes n})$, where $H$ is an involutive
Hopf algebra. Other such invariants exist~\cite{Woronowicz:Solutions,
MurakamiVanDerVeen:Quantized}. We expect that they are closely related.
\end{remark}

\begin{remark} \label{rem:core} It follows from the reasonings of
Section~\ref{sec:virtual} that it is possible to extract a maximal
braid out of an OU tangle, leaving behind a minimal ``core''
tangle. Precisly, if a virtual OU tangle $T$ is decomposed as $T=\beta'T'$
where $\beta'$ is a virtual pure braid and $T'$ is a virtual OU
tangle, and if $T'$ has the minimal possible crossing number for such
a decomposition, then $\beta'$ and $T'$ are uniquely detemined. Indeed,
let $(\beta',T')=f(I,T)$ be the final element guaranteed by the Diamond
Lemma (\ref{lem:diamond}) in the connected component of $(I,T)$ in
$\calX$. For example, if $T$ is $T_1$ of Example~\ref{exa:indivisible},
then $\beta'=\sigma_{12}$ and $T'$ is $T_2$ of~\ref{exa:indivisible}.

We do not know if the same is true for arbitrary virtual and/or classical
tangles. \endpar{\ref{rem:core}}
\end{remark}

\begin{discussion} \label{disc:EG}
There is a lovely visual side to the tools developed
for the proof of Chterental's Theorem (\ref{thm:vinj}). Given a reduced virtual OU tangle
$T\in\vcalROU_n$, we can consider the part $EG(T)$ of $\calX_n$ that
lies ``below'' $(I,T)$:
\[ EG(T) \coloneqq \left\{(\beta',T')\colon\, (I,T)\toto(\beta',T')\right\}. \]

\Needspace{48mm} 
\parpic[r]{
  \def\s#1#2{\,$\sigma_{#1#2}$} \def\si#1#2{\,$\sigma_{#1#2}^{-1}$}
  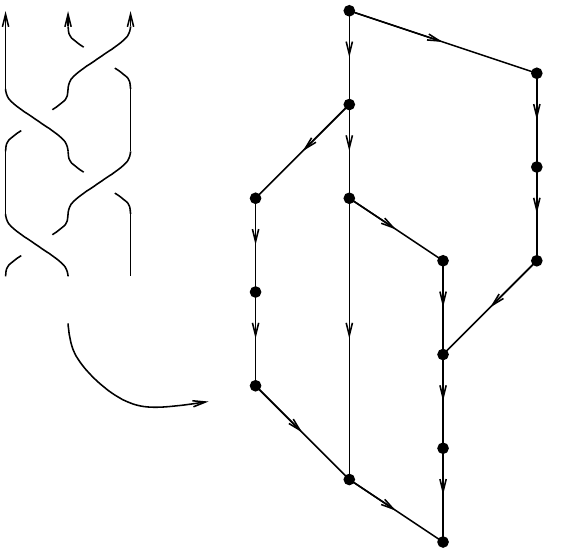
}
We restrict the relation $\to$ to $EG(T)$, making it into a directed
graph that we name ``the extraction graph of $T$''. By first computing
$\barGamma\circ\bariota$ or $\Ch=\barGamma_v\circ\bariota_v$, we
can also define $EG(\beta)$ when $\beta$ is a braid or a virtual
braid. These graphs are in themselves invariants (defined on $\vcalOU_n$
or $\vcalPB_n$ or $\calB_n$). They are often visually pleasing: we
have already seen a few examples, in Examples~\ref{exa:indivisible}
and~\ref{exa:GarsideHexagon}, within Equation~\ref{eq:CinnamonRolls},
and in Figure~\ref{fig:Cases23}\footnotemark.  Another
example, the extraction graph of the classical braid
$\sigma_{21}^{-1}\sigma_{13}\sigma_{32}^{-1}\sigma_{21}$ whose closure is
the figure-8 knot, is here on the right (we label edges by the relevant
divisor $\sigma_{ij}^{\pm 1}$ and vertices by the value of $\xi$). Some
even nicer examples appear in Section~\ref{ssec:EG}.

\footnotetext{Figure~\ref{fig:SCR} is not example because it misses a
part of the graph. See Section~\ref{ssec:EG}.}

For any $T$, $EG(T)$ is a finite graph (for the set of potential divisors
$\{\sigma_{ij}^{\pm 1}\}$ is finite and and only finitely many divisions
can be carried out before we run out of crossings). $EG(T)$ always has
an ``initial'' vertex $i$ (the pair $(I,T)$) and a final vertex $f$
--- the final element that is guaranteed by the Diamond Lemma and that
is discussed in Remark~\ref{rem:core}. Every vertex $v$ of $EG(T)$
is sandwiched between the two: $i\toto v\toto f$. Every ``wedge'' in
$EG(T)$ (Definition~\ref{def:diamond}) can be completed to a diamond of
one of the types appearing in Figures \ref{fig:SCR}, \ref{fig:Cases23},
\ref{fig:CaseM}, and~\ref{fig:Square} (hence all cycles in $EG(T)$ are
of even length, and hence $EG(T)$ is bipartite). If one travels from $i$
to $f$ along any path in $EG(T)$ while reading the generators indicated
on the edges, one always reads the same virtual braid.

If $T$ is $\barGamma(\iota(\beta))$ or $\Ch(\beta)$, the final
vertex $f$ of $EG(\beta)$ is $(\beta,I)$, and every path from $i$
to $f$ spells a braid word for $\beta$. Thus $EG(\beta)$ highlights
a finite set of ``special'' braid words for $\beta$. It follows from
Remark~\ref{rem:classical} that if $\beta$ is classical then all the
special words for it are classical too.

We don't really understand $EG(\beta)$ --- we don't know what properties
of $\beta$ can be read off $EG(\beta)$, and we don't know how to
characterize the ``special words'' for $\beta$ that appear in $EG(\beta)$
other than by repeating the definitions. \endpar{\ref{disc:EG}}

\end{discussion}

\draftcut {
\def\face{\begin{picture}(0,0)%
\includegraphics{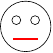}%
\end{picture}%
%
%
\setlength{\unitlength}{789sp}%
\begingroup\makeatletter\ifx\SetFigFont\undefined%
\gdef\SetFigFont#1#2#3#4#5{%
  \reset@font\fontsize{#1}{#2pt}%
  \fontfamily{#3}\fontseries{#4}\fontshape{#5}%
  \selectfont}%
\fi\endgroup%
\begin{picture}(1240,1240)(-19,-381)
\end{picture}%
}
\def\human{\begin{picture}(0,0)%
\includegraphics{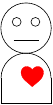}%
\end{picture}%
%
%
\setlength{\unitlength}{789sp}%
\begingroup\makeatletter\ifx\SetFigFont\undefined%
\gdef\SetFigFont#1#2#3#4#5{%
  \reset@font\fontsize{#1}{#2pt}%
  \fontfamily{#3}\fontseries{#4}\fontshape{#5}%
  \selectfont}%
\fi\endgroup%
\begin{picture}(1244,2442)(-21,-1583)
\end{picture}%
}
\def\machine{\begin{picture}(0,0)%
\includegraphics{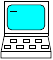}%
\end{picture}%
%
%
\setlength{\unitlength}{789sp}%
\begingroup\makeatletter\ifx\SetFigFont\undefined%
\gdef\SetFigFont#1#2#3#4#5{%
  \reset@font\fontsize{#1}{#2pt}%
  \fontfamily{#3}\fontseries{#4}\fontshape{#5}%
  \selectfont}%
\fi\endgroup%
\begin{picture}(1244,1394)(-21,-608)
\end{picture}%
}

\def\cellscale{1}
\newsavebox\pdfbox
\newlength{\pdfheight}

\def\nbpdfInput#1{{%
  \savebox{\pdfbox}{\includegraphics[scale=\cellscale]{#1}}%
  \settoheight{\pdfheight}{\usebox{\pdfbox}}%
  \noindent\imagetop{\ifdim\pdfheight<10mm\begin{picture}(0,0)%
\includegraphics{figs/face.pdf}%
\end{picture}%
%
%
\setlength{\unitlength}{789sp}%
\begingroup\makeatletter\ifx\SetFigFont\undefined%
\gdef\SetFigFont#1#2#3#4#5{%
  \reset@font\fontsize{#1}{#2pt}%
  \fontfamily{#3}\fontseries{#4}\fontshape{#5}%
  \selectfont}%
\fi\endgroup%
\begin{picture}(1240,1240)(-19,-381)
\end{picture}%
\else\begin{picture}(0,0)%
\includegraphics{figs/human.pdf}%
\end{picture}%
%
%
\setlength{\unitlength}{789sp}%
\begingroup\makeatletter\ifx\SetFigFont\undefined%
\gdef\SetFigFont#1#2#3#4#5{%
  \reset@font\fontsize{#1}{#2pt}%
  \fontfamily{#3}\fontseries{#4}\fontshape{#5}%
  \selectfont}%
\fi\endgroup%
\begin{picture}(1244,2442)(-21,-1583)
\end{picture}%
\fi}\ %
  \imagetop{\usebox{\pdfbox}}%
  \vskip 2mm%
}}

\def\nbpdfOutput#1{{\noindent{\imagetop{\begin{picture}(0,0)%
\includegraphics{figs/machine.pdf}%
\end{picture}%
%
%
\setlength{\unitlength}{789sp}%
\begingroup\makeatletter\ifx\SetFigFont\undefined%
\gdef\SetFigFont#1#2#3#4#5{%
  \reset@font\fontsize{#1}{#2pt}%
  \fontfamily{#3}\fontseries{#4}\fontshape{#5}%
  \selectfont}%
\fi\endgroup%
\begin{picture}(1244,1394)(-21,-608)
\end{picture}%
}\ \imagetop{\includegraphics[scale=\cellscale]{#1}}\vskip 2mm}}}

\def\m#1{\text{\tt #1}}

\section{Some Computations} \label{sec:comp}

When mathematics is computable, we feel it is appropriate and necessary to include an implementation. In this case, the implementation is concise and follows the notation and logical structure of a paper, so we choose to include it as an integral part of that paper. We hope that the programs presented here serve as an illustration of the overall simplicity and validity of the ideas within the paper, and that they encourage others to play and further discover. The implementations also lead to new enumerations -- the tables in Sections~\ref{ssec:VB}~and~\ref{ssec:CB} -- and to intriguing and appealing graph-valued invariants -- Section~\ref{ssec:EG} -- which cannot be computed otherwise, and which may lead to further study.

All code here is written in {\sl Mathematica}~\cite{Wolfram:Mathematica} and is available as the {\sl Mathematica} notebook {\sl SomeComputations.nb} at~\cite{Self}.

\parpic[r]{\begin{picture}(0,0)%
\includegraphics{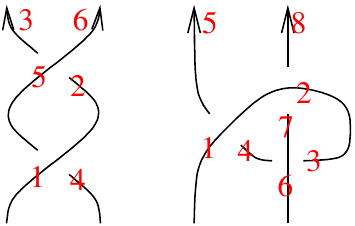}%
\end{picture}%
%
%
\setlength{\unitlength}{3947sp}%
\begingroup\makeatletter\ifx\SetFigFont\undefined%
\gdef\SetFigFont#1#2#3#4#5{%
  \reset@font\fontsize{#1}{#2pt}%
  \fontfamily{#3}\fontseries{#4}\fontshape{#5}%
  \selectfont}%
\fi\endgroup%
\begin{picture}(1693,1074)(-30,-223)
\end{picture}%
}
\subsection{Implementing virtual OU tangles, virtual braids, and $\Ch$}
To represent a virtual tangle diagram $D$ on the computer, we order its strands and traverse each of them in order, marking each ``O'' point, each ``U'' point, and each end of strand, with the integers $1,2,3,\ldots$, in the order in which they are encountered. See examples on the right. For each crossing $x$ of $D$ we form a {\sl Mathematica} expression $\m{X}_s[i,j]$, where $s$ is the sign of the crossing and $i$ and $j$ are the markings next to the O side and the U side of $x$, respectively. We also form an expression $\m{EOS}[k]$ for each end-of-strand marked $k$. We toss all this information into a container \m{VD}, and the result is our computer representation of $D$. Below, \m{vd1} and \m{vd2} are the results of this process for the two example tangles shown here.

\noindent\nbpdfInput{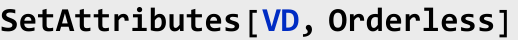}

\noindent\nbpdfInput{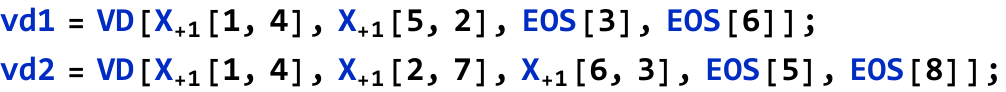}

Sometimes in a \m{VD} we allow to label O/U/\m{EOS} points by arbitrary real numbers, for in fact, only the ordering of these points matter. The routine \m{Tidy} takes a real-ordered \m{VD} and converts it to a sequentially ordered one. Thus it brings a \m{VD} to a ``canonical form'':

\noindent\nbpdfInput{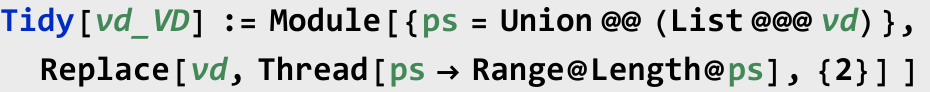}

\noindent\nbpdfInput{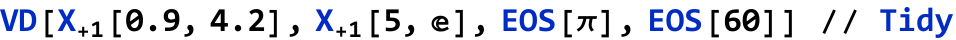}

\noindent\nbpdfOutput{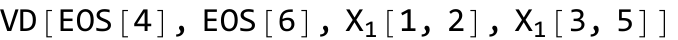}

The routine \m{R12Reduce1} reduces a virtual diagram by performing one R2 or R1 move, if such a move is available, and otherwise it does nothing. The routine \m{R12Reduce} finds the fixed point of \m{R12Reduce1} --- in other words, it reduces a virtual diagram using all available R1 and R2 moves.

\noindent\nbpdfInput{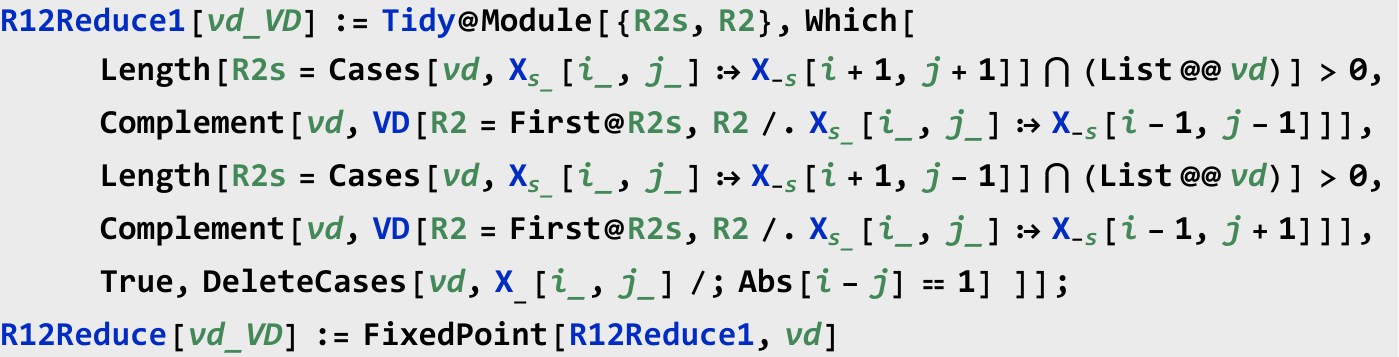}

Here's a very minor example:

\noindent\nbpdfInput{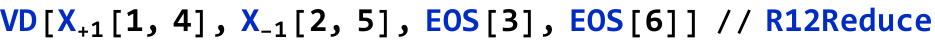}

\noindent\nbpdfOutput{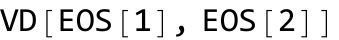}

\Needspace{30mm} 
\parpic[r]{
  \def\Xa{$\m{X}_{s_1}[i_1,j_1]$}   \def\Xb{$\m{X}_{s_2}[i_2,j_2]$}
  \def\Xc{$\m{X}_{s_2}[j_1,j_2]$}   \def\Xd{$\m{X}_{s_1}[i_1,i_2]$}
  \def\Xe{$\m{X}_{s_1s_2}[i_1\!-\!s_1/3,j_2\!+\!s_2/3]$}   \def\Xf{$\m{X}_{-s_1s_2}[i_1\!+\!s_1/3,j_2\!-\!s_2/3]$}
  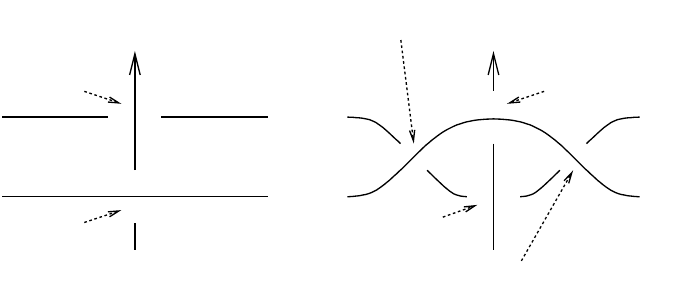}
In a similar manner, \m{$\Gamma$1} performs one glide move if one is available, and \m{$\barGamma$} fully reduces under both glide moves and R1 and R2 moves. Here we bound the number of iterations by $2^{24}$, to artificially stop runaway reductions such as the one in Figure~\ref{fig:swirls}.

\noindent\nbpdfInput{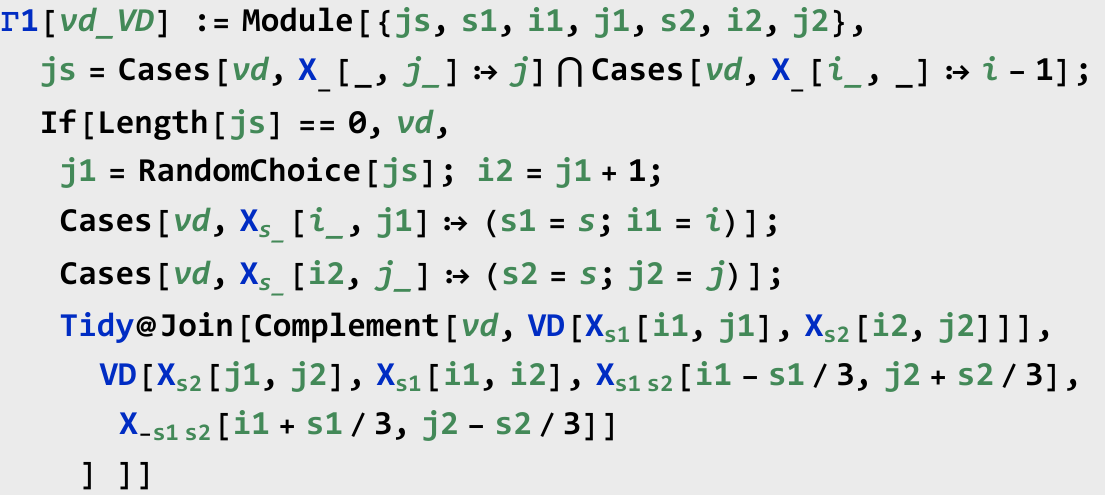}

\noindent\nbpdfInput{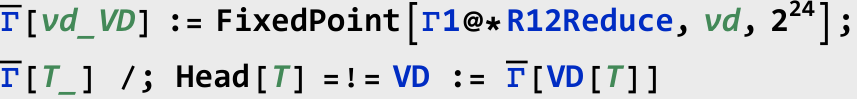}

As expected, $\barGamma(\m{vd1})=\m{vd2}$:

\noindent\nbpdfInput{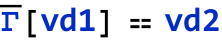}

\noindent\nbpdfOutput{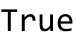}

Next we define the composition operation \m{d1**d2} of virtual tangle diagrams. The implementation works by ``shrinking'' \m{d2} so that each of its strands would fit between the last crossing in the corresponding strand of \m{d1} and the \m{EOS} at the end of that strand of \m{d1}, then taking the union of \m{d1} and the shrank \m{d2}, and then applying \m{Tidy} to the result:

\noindent\nbpdfInput{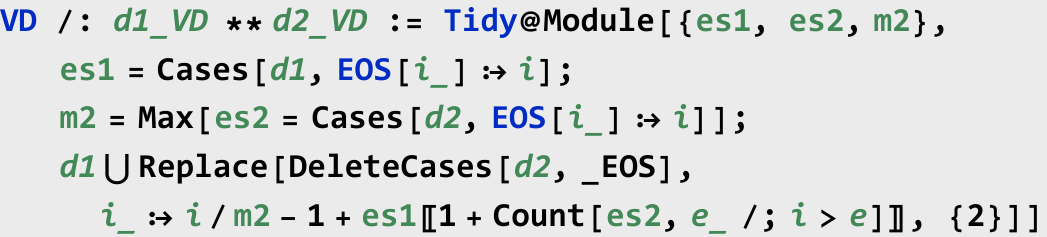}

For example, ``our'' \m{vd2} has $3$ crossings yet is equivalent to a 2-twist braid. So $\m{vd1}\cdot\m{vd2}$ ought to have $6$ crossings while its reduced OU form, $\barGamma(\m{vd1}\cdot\m{vd2})$ should be the Cinnamon Roll $\CR_4$, which has $7$ crossings. The computer agrees:

\noindent\nbpdfInput{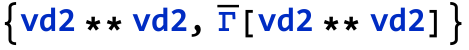}

\noindent\nbpdfOutput{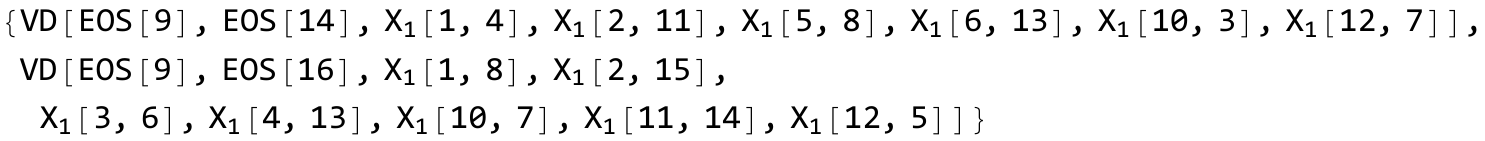}

Next we implement virtual pure braids, and it is best to start with an example. We represent the 3-strand virtual pure braid $\beta=\sigma_{21}^{-1}\sigma_{13}\sigma_{31}\sigma_{13}\sigma_{31}\sigma_{13}\sigma_{23}\sigma_{21}$ of Example~\ref{exa:SCR} by the {\sl Mathematica} expression below:

\noindent\nbpdfInput{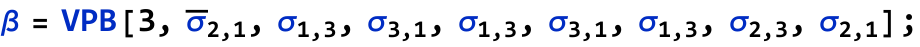}

The conversion of \m{VPB}s into \m{VD}s is quite easy. We just need to define it on the generators and then use the already-available composition of \m{VD}s to extend the definition to products of generators:

\noindent\nbpdfInput{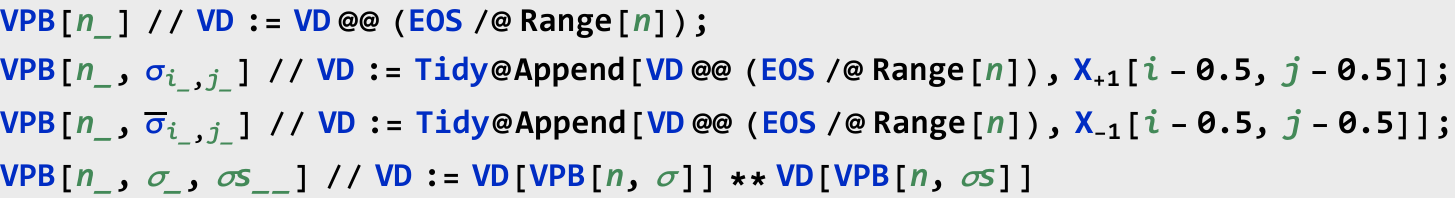}

We can compute $\Ch(\beta)=\barGamma_v(\bariota_v(\beta))$ (count that it has 18 \m{X} symbols, just as Figure~\ref{fig:SCR}~(A) has 18 crossings!):

\noindent\nbpdfInput{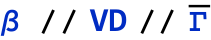}

\noindent\nbpdfOutput{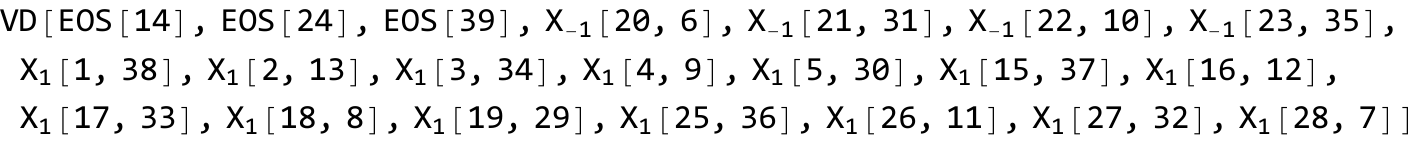}

We can even verify Equation~\eqref{eq:SCRBraids1}:

\noindent\nbpdfInput{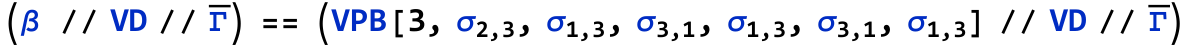}

\noindent\nbpdfOutput{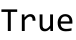}

\subsection{Tabulating Virtual Pure Braids} \label{ssec:VB}
Our next task is to tabulate virtual pure braids with a given number of strands $n$ and a bound $m$ on the number of crossings. The first routine, \m{VPBGens}, outputs the list of all generators of $\vcalPB_n$:

\noindent\nbpdfInput{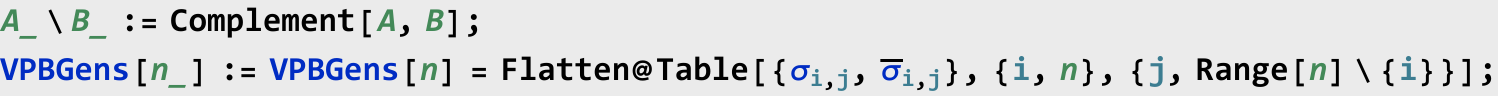}

\noindent\nbpdfInput{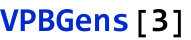}

\noindent\nbpdfOutput{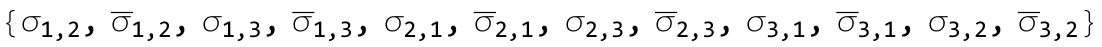}

Next we'd like to generate all words in the generators we just computed, and separate them using $\Ch$ and Chterental's Theorem (\ref{thm:vinj}). To save some computer effort, we generate only ``proud'' words --- words that do not contain a letter followed by its inverse, or adjacent commuting letters that are not in lexicographic order. The ``Proud Followers'' \m{PF} of a generator are those generators that can follow it without ruining the pride of a word:

\noindent\nbpdfInput{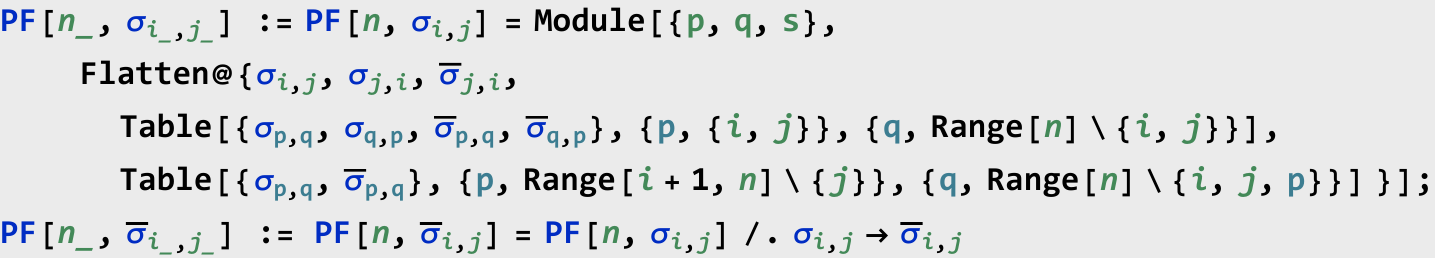}

\noindent\nbpdfInput{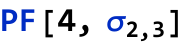}

\noindent\nbpdfOutput{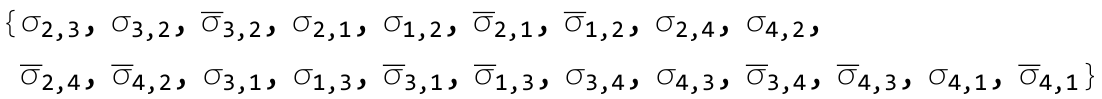}

And then $\m{PVPBDs}[n,m]$ computes all Proud Virtual Pure Braid Diagrams on $n$ strands and with $m$ crossings:

\noindent\nbpdfInput{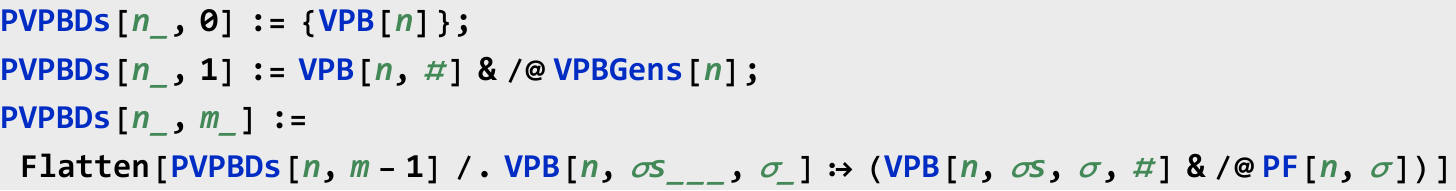}

\noindent\nbpdfInput{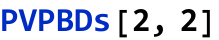}

\noindent\nbpdfOutput{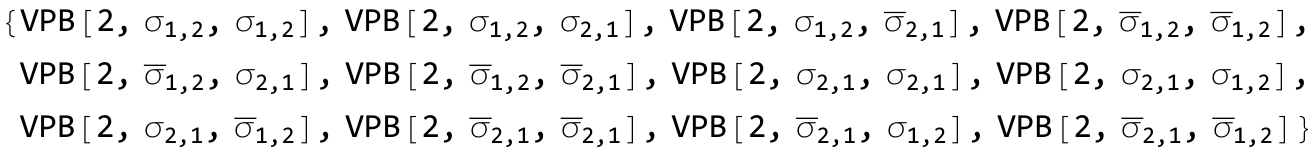}

These sets grow very rapidly:

\noindent\nbpdfInput{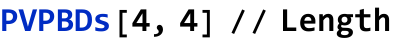}

\noindent\nbpdfOutput{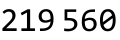}

$\m {AllVPBs}[n, m] $ finds representatives for all virtual braids on $n$ strands with at most $m$ crossings, by using $\m{PVPBDs}[n,m]$ and then deleting duplicates by $\barGamma_v$:

\noindent\nbpdfInput{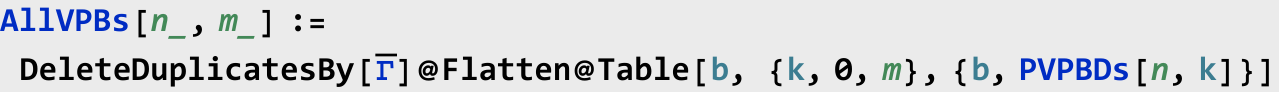}

\noindent\nbpdfInput{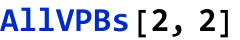}

\noindent\nbpdfOutput{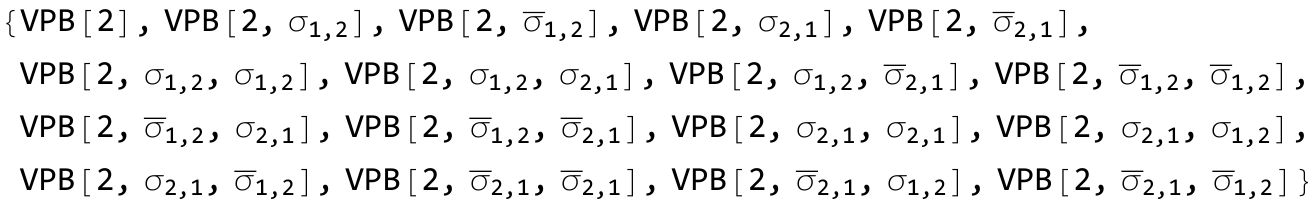}

There are 15,156 virtual pure braids with 3 strands and precisely 4 crossings (meaning, braids in $\m {AllVPBs}[3,4] $ but excluding those in $\m {AllVPBs}[3,3]$). It took our computer about 86 seconds to figure that out:

\noindent\nbpdfInput{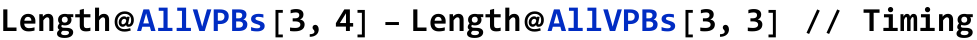}

\noindent\nbpdfOutput{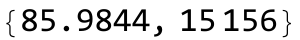}

\Needspace{65mm} 
\parpic[r]{\setlength{\tabcolsep}{3pt}\begin{tabular}{|c|c|c|c|c|c|}
\hline
$m\backslash n$	&	2 &		3 &			4 &			5 &			6 \\
\hline
0				&	1 &		1 &			1 &			1 &			1 \\
1				&	4 &		12 &			24 &			40 &			60 \\
2				&	12 &		132 &		504 &		1,320 &		2,820 \\
3				&       36 &		1,416 &		10,344 &             	41,760 & 		124,140 \\
4				&       108 &	15,156 &		211,416 &            1,308,360 &	5,357,700 \\
5				&       324 &	162,156 &	4,317,912 &	&			\\
6				&       972 &	1,734,864 &	&			&			\\
\hline
\end{tabular}}

In our spare time we have tabulated the numbers of $n$-strand pure virtual braids with precisely $m$ crossings for some small values of $n$ and $m$. The results are on the right, and data files containing the actual braids are at~\cite{Self}. As a test of the integrity of our programs we also computed most of the numbers in this table by generating all braid words and reducing modulo all relations\footnotemark. The numbers match. See the {\sl Mathematica} notebook {\sl VPBByGensAndRels.nb} at~\cite{Self}.
\footnotetext{Sometimes two braid words of length $m_1$ are related by a chain of relations that pass through words of length $m_2$, where $m_2>m_1$, and we do not know in advance a bound on $m_2$. Hence the computation using generators and relations is slow (as we have to raise $m_2$ and the number of words to consider grows very big) and unreliable (strictly speaking, we only get upper bounds on the braid counts).}

\subsection{Tabulating Classical Braids} \label{ssec:CB}
It is a bit odd that we have not seen a table such as the one above, but for classical braids. As the classical braid group is automatic~\cite{Epstein:WordProcessing} and hence the word problem in it is very easy, there are much better in-theory tools than ours to produce such a table. Yet our tools are implemented in practice, and we may as well use them.

First, we need to be able to convert from a standard classical braid notation~\cite{Bar-NatanMorrison:KnotTheory} to the \m{VPB} notation used here.

\noindent\nbpdfInput{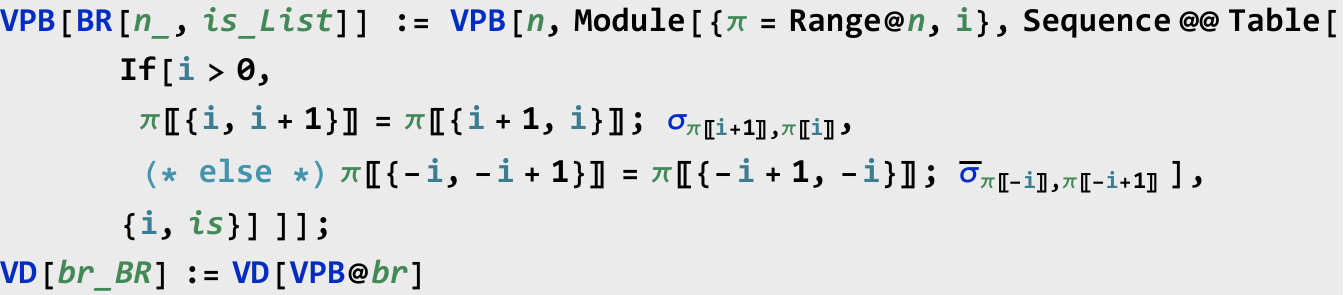}

\noindent\nbpdfInput{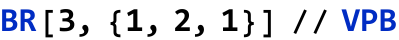}

\noindent\nbpdfOutput{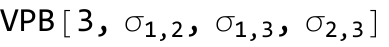}

After that, we repeat the same steps as in the virtual case:

\noindent\nbpdfInput{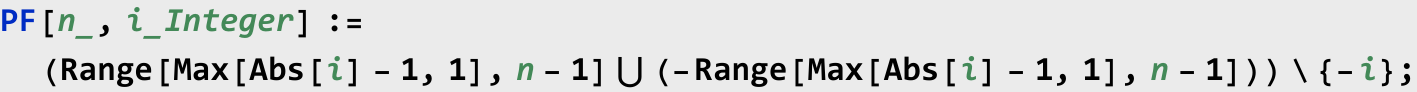}

\noindent\nbpdfInput{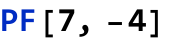}

\noindent\nbpdfOutput{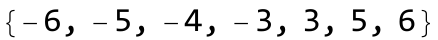}

\noindent\nbpdfInput{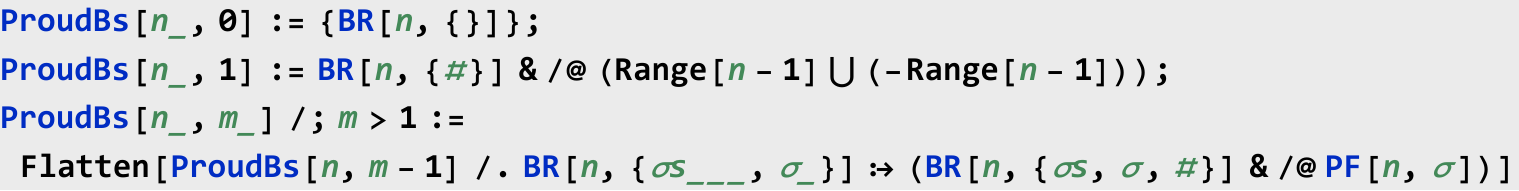}

\noindent\nbpdfInput{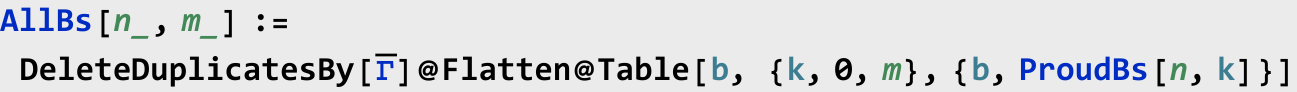}

For example, here are all the distinct positive 3-strand braids:

\noindent\nbpdfInput{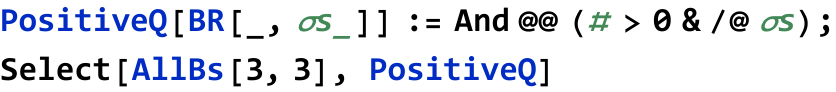}

\noindent\nbpdfOutput{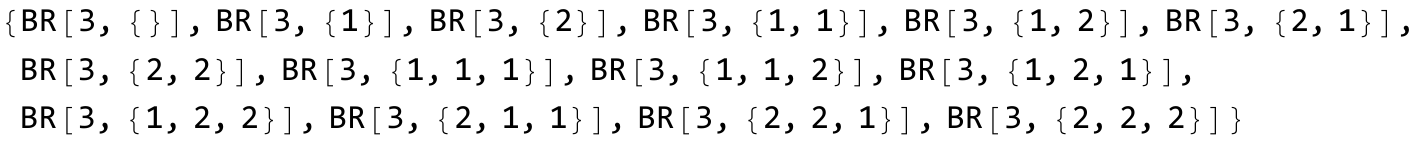}

On our computer, it takes about 20 seconds to find that there are 1,110 classical braids with 4 strands and crossing number equal to 5:

\noindent\nbpdfInput{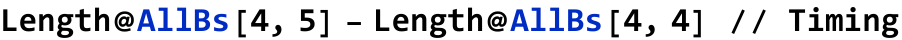}

\noindent\nbpdfOutput{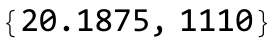}

\parpic[r]{\setlength{\tabcolsep}{3pt}\begin{tabular}{|c|c|c|c|c|c|}
\hline
$m\backslash n$	&	2 &	3 &		4 &		5 &		6		\\
\hline
0				&	1 &	1 &		1 &		1 &		1		\\
1				&	2 &	4 &		6 &		8 &		10		\\
2				&	2 &	12 &		26 &		44 &		66		\\
3				&	2 &	30 &		98 &		206 &	362		\\
4				&	2 &	68 &		338 &	884 &	1,794	\\
5				&	2 &	148 &	1,110 &	3,600 &	8,370	\\
6				&	2 &	314 &	3,542 &	14,198 &	37,606	\\
7				&	2 &	656 &	11,098 &	54,876 &	164,910	\\
8				&	2 &	1,356 &	34,362 &	209,348 &	711,746	\\
9				&	2 &	2,782 &	105,546 &	791,798 &	3,039,546	\\
\hline
\end{tabular}}
And here's a table of the numbers of $n$-strand pure virtual braids with precisely $m$ crossings, for small values of $n$ and $m$. The data files containing the actual braids are at~\cite{Self}.

Note that the entries in the $n=3$ column of this table fit with the sequence $6\cdot 2^m-2F_{m+3}-2$, where $F_m$ is the $m$th Fibonacci number:

\noindent\nbpdfInput{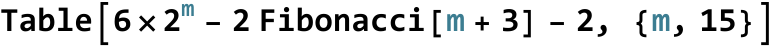}

\noindent\nbpdfOutput{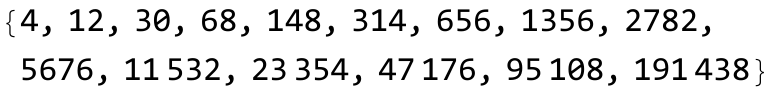}

\picskip{0}
The fit persists at least up to $m=12$. We do not know why this is so.

\subsection{Extraction Graphs} \label{ssec:EG}
We can now write a short program \m{EG}, to compute and display Extraction Graphs as in Discussion~\ref{disc:EG}.

\noindent\nbpdfInput{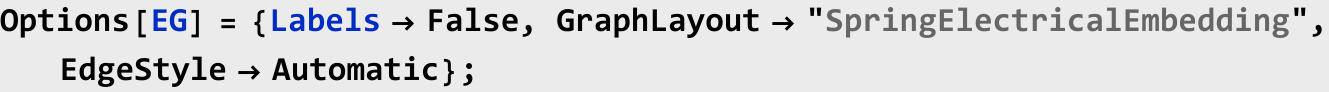}

\noindent\nbpdfInput{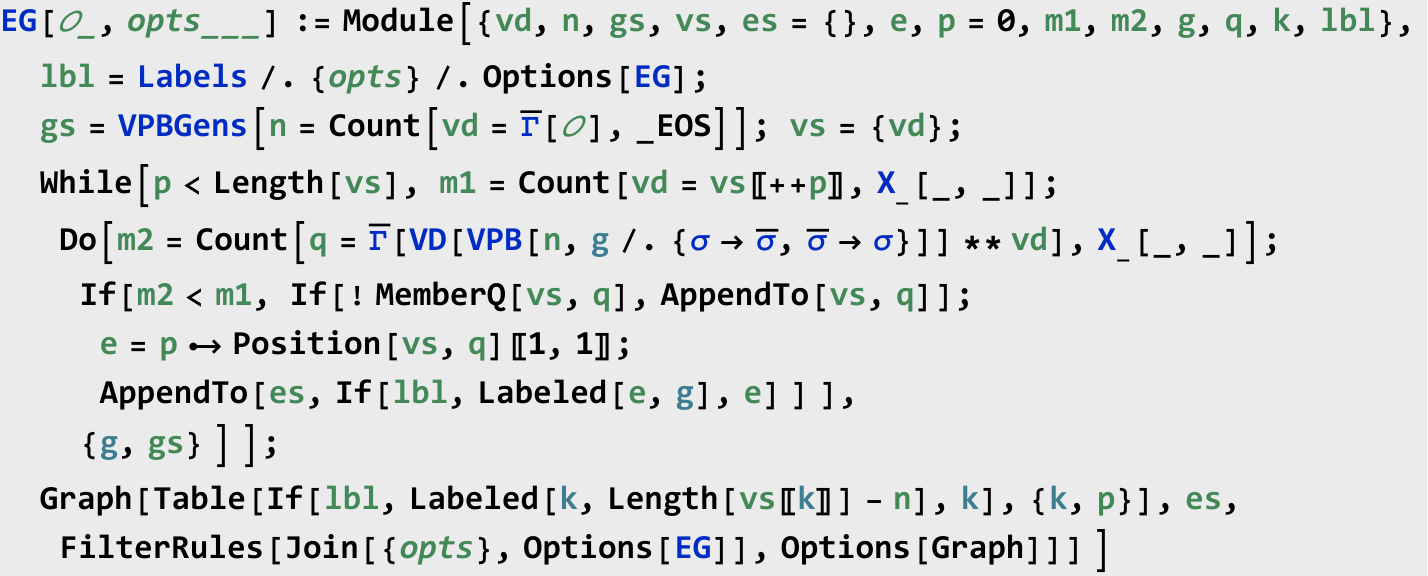}

Note that the diamond in Figure~\ref{fig:SCR} is genuine, but it is not an extraction graph, because the full extraction graph of the initial OU tangle of that figure contains two further edges:

\noindent\nbpdfInput{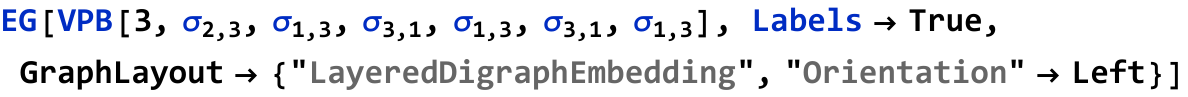}

\noindent\nbpdfOutput{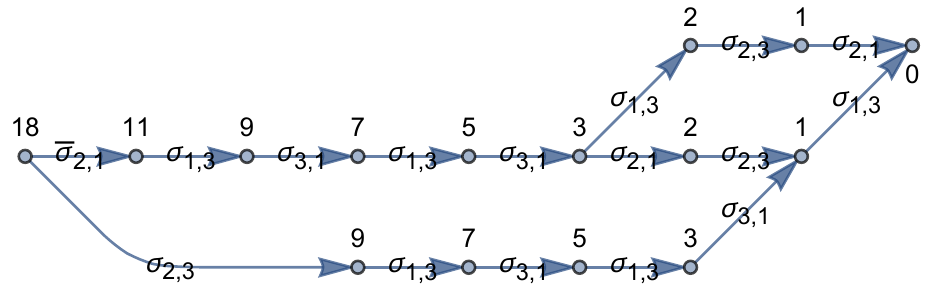}

The braid below, suggested to us by B.~Wiest, has a linear extraction graph and hence a unique ``special word'' (see Discussion~\ref{disc:EG}), but that word is of length 13, whereas the braid can be presented by a shorter word $\beta$, of length 11:

\noindent\nbpdfInput{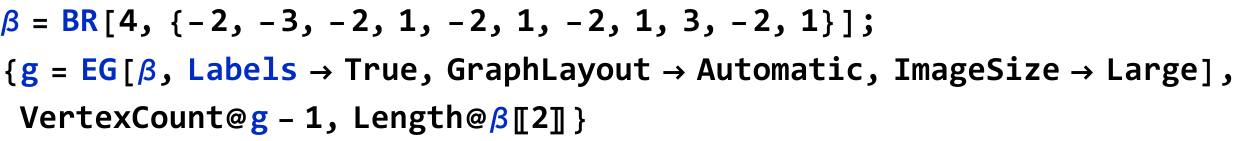}

\noindent\nbpdfOutput{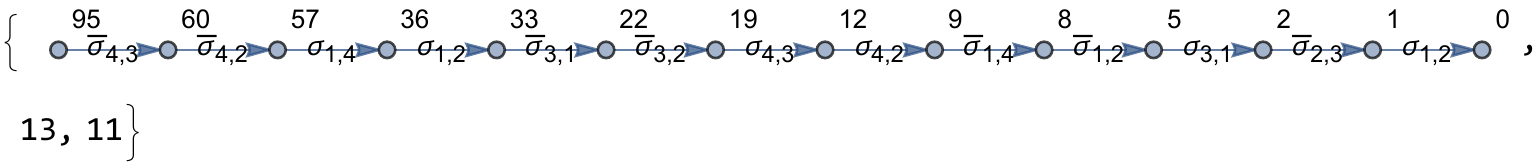}

It is easy to see that the extraction graph of the 4-crossing 8-strand braid $\overcrossing\ \overcrossing\ \overcrossing\ \overcrossing$ is the tesseract:

\noindent\nbpdfInput{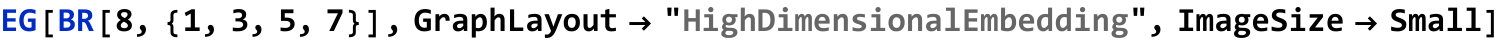}

\noindent\nbpdfOutput{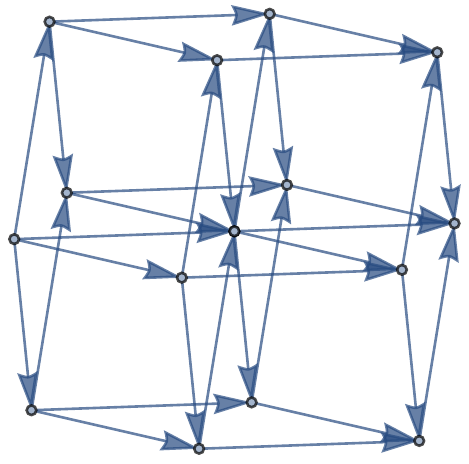}

The extraction graphs of Garside braids seem to be permutahedra (we did not attempt to prove this in general):

\noindent\nbpdfInput{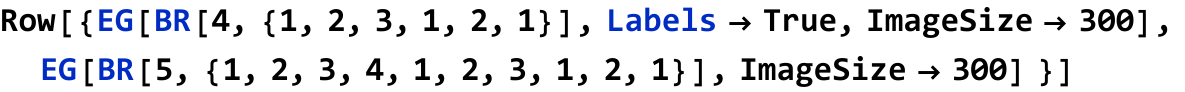}

\noindent\nbpdfOutput{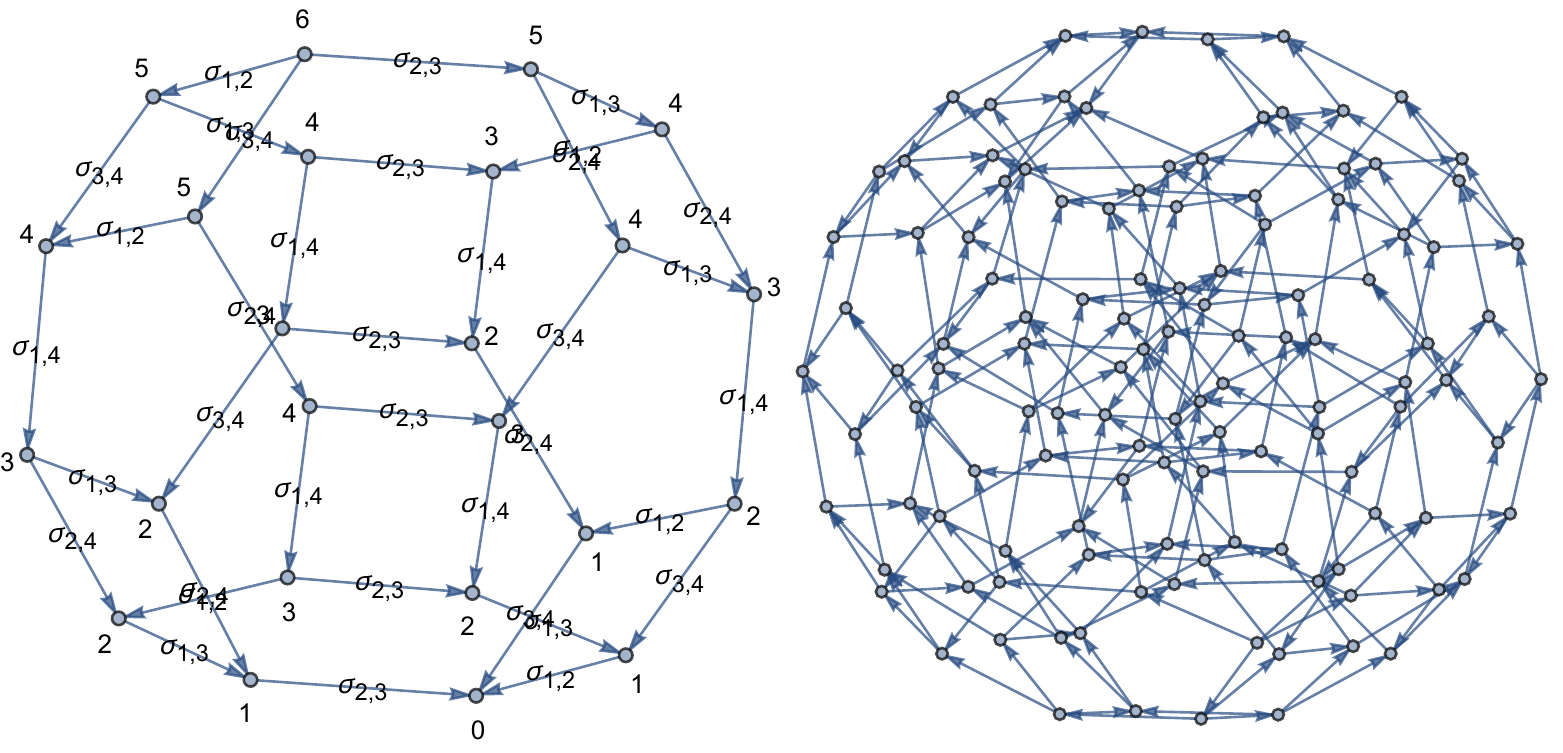}

Sometimes extraction graphs can be amusing. In no particular order, here are a lifesaver, an impressionistic map of the US state of Iowa, a torch flame, a legless bird, a feather, a ladder, a tennis racket, and a mouse trap:

\noindent\nbpdfInput{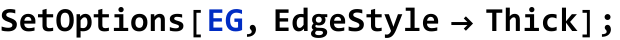}

\noindent\nbpdfInput{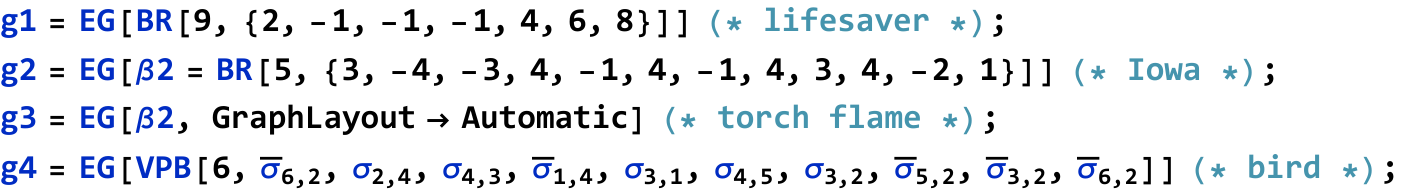}

\noindent\nbpdfInput{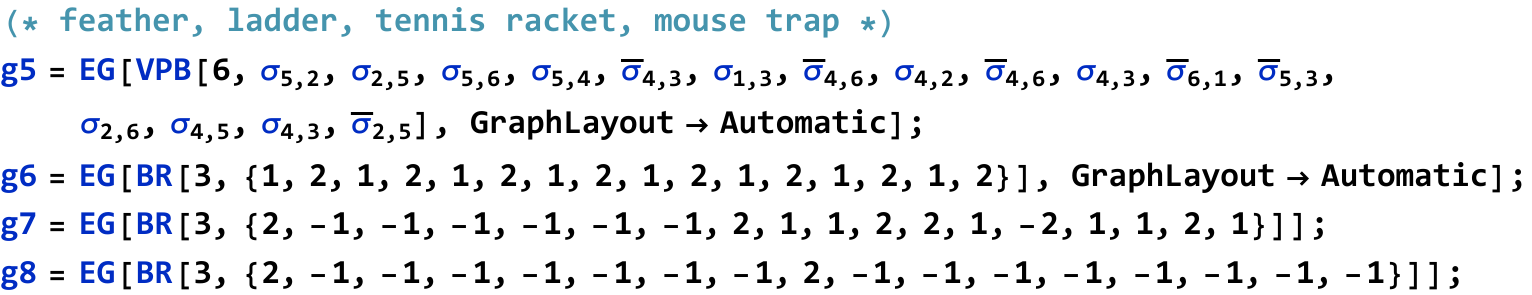}

\noindent\nbpdfInput{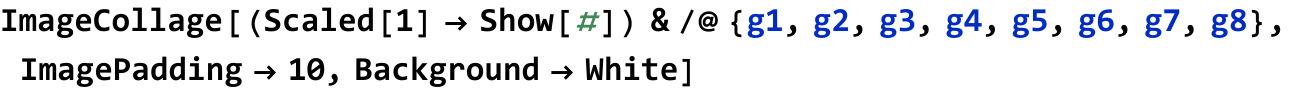}

\noindent\nbpdfOutput{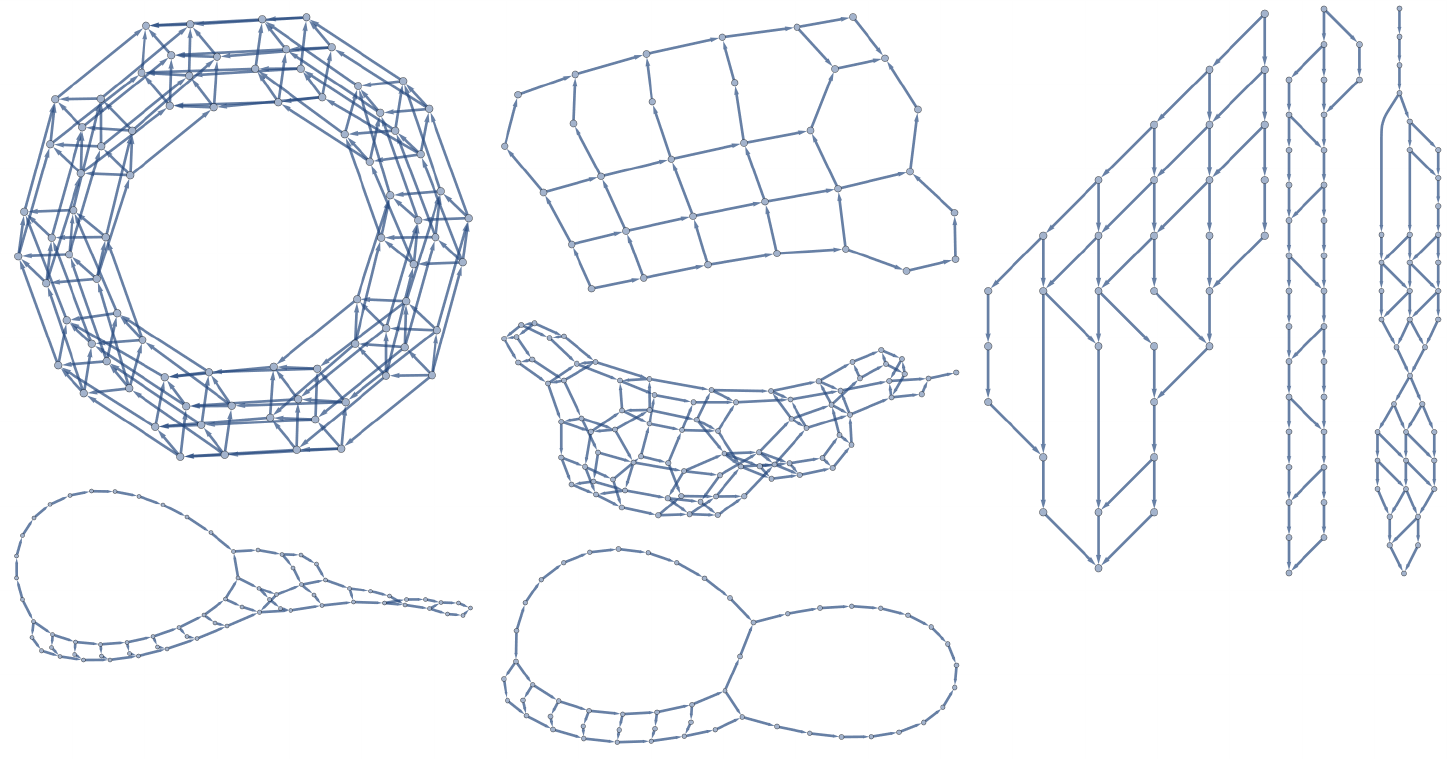}

We don't know what, if any, can be learned about braids from these graphs, and we can only hope the referee will forgive us for having a bit of fun.

\subsection{Computational Complexity} \label{ssec:CC}
Looking again at Figure~\ref{fig:divquo}~(C), we see that in the worst case, if the crossing number $\xi(T)$ of an OU tangle $T$ is $p$, the crossing number of the OU version of $\sigma_{ij}^{\pm 1}T$ might be as big as $3p+1$, and hence the complexity of computing $\Ch$ grows exponentially. Here are the ``worst'' classical and virtual braids with 8 crossings. A bit more is in the {\sl Mathematica} notebook {\sl TheWorstBraids.nb} at~\cite{Self}.

\noindent\nbpdfInput{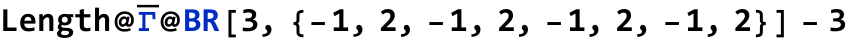}

\noindent\nbpdfOutput{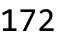}

\noindent\nbpdfInput{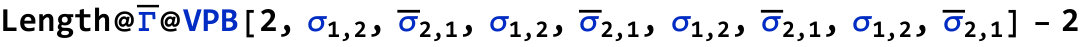}

\noindent\nbpdfOutput{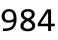}

}

\draftcut \section{There's more!} \label{sec:more}

There's more! In fact, OU tangles and OU ideas seem prevalent in knot
theory, even though it seems that nobody collected all these ideas
together before. In fact, possibly the most
important contribution of this paper is the observation that, in addition to the detailed examples that we studied throughout, everything mentioned in
this section is OU-related.

\subsection{Weakening the Bond}

The Gliding Fheorem (\ref{fhm:every}) fails because the bond between the strands of
a single crossing is too strong; they cannot be separated to be taken
for rides along other strands in an independent manner: when the U and the
O of a UO interval belong to the same crossing, one cannot glide them
independently of each other and across each other as the glide move of
Figure~\ref{fig:Gliding} dictates. So we seek to weaken this bond.

\Needspace{39mm} 
\parpic[r]{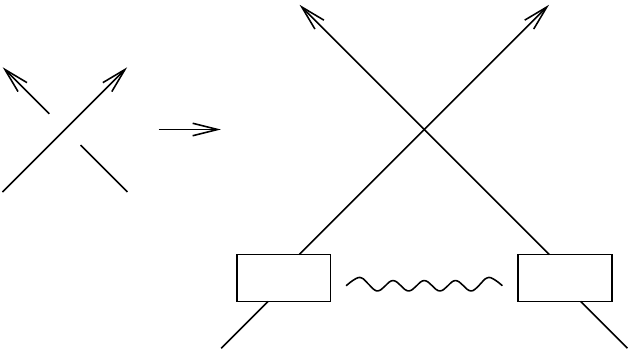}
One way to do so is with algebra. One aims to construct invariants of
tangles by placing ``R-matrices'' on positive crossings (and their
inverses on negative crossings). An R-matrix is an element
$R=\sum_i{b_i\otimes a_i}\in H\otimes H$ in the tensor square of some
algebra $H$, and its $b_i$ side is placed on the O side of the crossing
while its $a_i$ side is put on the U side. This done, one multiplies the
algebra elements seen on each strand in the order in which they appear
along it, and the hope is that the result would be an invariant of the
tangle, living in $H^{\otimes S}$ where $S$ is the set of strands.

In this context ``O'' becomes ``$b_i$'' and ``U'' becomes ``$a_i$'',
and the bond between O and U is nearly severed --- within a long
product, given the appropriate commutation relations, $b_i$'s can
be commuted against $a_i$'s whether or not they originally came from
the same crossing. Further effort is needed in order to make use of
this fact, and it is beyond the scope of this summary to reproduce
this effort here. Yet the result becomes ``something from nothing'':
given relatively little input, a construction of an R-matrix and the
algebra $H$ in which it lives. This construction is better known as
``the Drinfel'd double construction''. See more at~\cite{Talk:OU} and
hopefully in a future publication.

\Needspace{40mm} 
\parpic[r]{\begin{picture}(0,0)%
\includegraphics{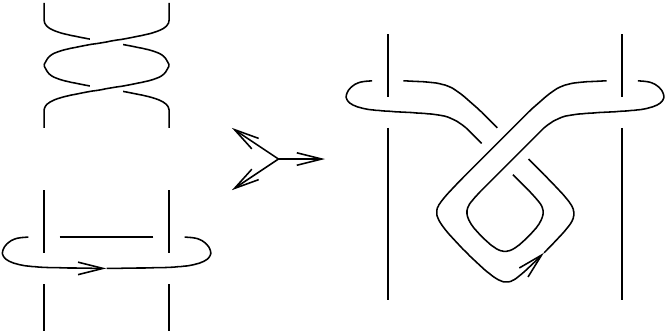}%
\end{picture}%
%
%
\setlength{\unitlength}{3947sp}%
\begingroup\makeatletter\ifx\SetFigFont\undefined%
\gdef\SetFigFont#1#2#3#4#5{%
  \reset@font\fontsize{#1}{#2pt}%
  \fontfamily{#3}\fontseries{#4}\fontshape{#5}%
  \selectfont}%
\fi\endgroup%
\begin{picture}(3199,1599)(14,-748)
\put(526,-661){\makebox(0,0)[b]{\smash{{\SetFigFont{10}{12.0}{\rmdefault}{\mddefault}{\updefault}{\color[rgb]{0,0,0}$+1$}%
}}}}
\put(2626,-511){\makebox(0,0)[lb]{\smash{{\SetFigFont{10}{12.0}{\rmdefault}{\mddefault}{\updefault}{\color[rgb]{0,0,0}$+1$}%
}}}}
\end{picture}%
}
Another way to weaken the bond between the O side and the U side of a single crossing is to represent crossings
using surgery. A quick summary is on the right: a crossing can be created using a $+1$ surgery on a loop
surrounding the two strands to be crossed, and that loop is relatively loose bond between these two strands,
for in itself it can be pushed around.

This story is imprecise and incomplete: Imprecise because strictly speaking, the surgery shown created two
crossings and not just one. Incomplete in several ways; the most important is that general surgeries can
change the ambient space from $S^3$ into another 3-manifold, and thus to properly pursue this idea one must
study an appropriate class of tangles in manifolds. See more at~\cite{ThurstonD:Sutured} and
hopefully in a future publication.

\subsection{Prior Art}
Milnor attempted to classify links up to link homotopy by introducing the
``reduced peripheral system''~\cite{Milnor:LinkGroups}. Unfortunately,
this is only a complete invariant of homotopy links with up to three
components.  In~\cite{AudouxMeilhan:PeripheralSystems} Audoux and Meilhan
use a gliding-type algorithm and OU links to give a full analysis of the
kernel of this invariant: namely, they prove that the reduced peripheral
system is a complete invariant of {\em welded} homotopy links (welded
links up to {\em self-virtualization} equivalence). Welded links are
closely related to knotted ribbon tori in $\bbR^4$, and homotopy links
map non-injectively into welded homotopy links by ``spinning around a
plane.'' See~\cite[Definition~1.7]{AudouxMeilhan:PeripheralSystems},
where OU links are called {\em sorted}. See
also~\cite[Definition~4.15]{AudoxBellingeriMeilhanWagner:WeldedStringLinks}
where they are called ``ascending''.

An earlier occurrence of OU ideas in the context of w-tangles is in the
paper~\cite{KBH} whose theme is the separation of hoops, that can only
go Under, from balloons, that go both Under and Over (so~\cite{KBH}
is a bit less ``pure'', as the balloons are not quite O). Later within
the same paper, and also within \cite{WKO2, WKO3, WKO4}, the associated
graded space of the space of w-tangles is studied, the space $\calA^w$
of ``arrow diagrams modulo the TC relation''. Furthermore that space is
studied using various ``Heads then Tails'' techniques, which in the language of the
current paper, correspond to UO presentations (not OU, but of course,
it's essentially the same). See especially~\cite[Section~2.4]{WKO4}.

\Needspace{15mm} 
\parpic[r]{\begin{picture}(0,0)%
\includegraphics{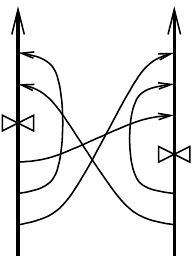}%
\end{picture}%
%
%
\setlength{\unitlength}{3947sp}%
\begingroup\makeatletter\ifx\SetFigFont\undefined%
\gdef\SetFigFont#1#2#3#4#5{%
  \reset@font\fontsize{#1}{#2pt}%
  \fontfamily{#3}\fontseries{#4}\fontshape{#5}%
  \selectfont}%
\fi\endgroup%
\begin{picture}(924,1244)(64,-758)
\end{picture}%
}
An even earlier occurrence of OU ideas, in the associated
graded $\calA^v$ context for virtual tangles, occurs in a
very well-hidden way within Enriquez' work on quantization
of Lie bialgebras~\cite{Enriquez:UniversalAlgebras,
Enriquez:QuantizationFunctors}. For example, his ``universal
algebras''~\cite[Section~1.3.2]{Enriquez:QuantizationFunctors} are
isomorphic to the space $\calA^v_{OU}$ of arrow diagrams as on the right,
in which all arrow tails occur before all arrow heads (that's OU!), and
is endowed with the product that $\calA^v_{OU}$ inherits from the stacking
product of $\calA^v$ (which is the analogue of the product used in our
paper). We are afraid that there aren't excellent introductions available
on $\calA^v$ and its relationship with virtual tangles. Hopefully we
will write one one day. Until then, some information is in~\cite{WKO2}
and in lecture series such as~\cite{Caen, MasterClass}. We also hope to
one day explain the Enriquez work as the construction of a ``homomorphic
expansion''~\cite{WKO1} for the space of virtual OU / acyclic tangles.

If $\frakg=\fraka^\star\bowtie\fraka$ is the double of a Lie bialgebra
$\fraka$, there is a standard interpretation of $\calA^v$ as a space
of formulas for elements in tensor powers $\calU(\frakg)^{\otimes n}$
of the universal enveloping algebra $\calU(\frakg)$ of $\frakg$. Within
this context, arrow tails (or ``O'') correspond to $\fraka^\star$
and arrow heads (or ``U'') correspond to $\fraka$, and the O then U
theme of this paper corresponds to the ``polarization'' isomorphism
$\calU(\frakg)\cong\calU(\fraka^\star)\otimes\calU(\fraka)$, which is
a consequence of the PBW theorem. In itself, the polarization isomorphism is central to all approaches to the
quantization of Lie bialgebras~\cite{EtingofKazhdan:BialgebrasI, Severa:BialgebrasRevisited}.

\draftcut \section{Acknowledgement} We wish to thank P.~Bellingeri,
A.~Referees, D.~Thurston, and B.~Wiest for comments and suggestions, and especially
O.~Chterental for spotting a major gap in an earlier version of this
paper. This work was partially supported by NSERC grant RGPIN-2018-04350 and by the University of Sydney Visiting Scholar Scheme.

\if\draft y
  \draftcut \input{tldt20.tex}
  \draftcut \input{recycling.tex}
\fi

\draftcut
\AtEndDocument{
\vfill
\[ \includegraphics[height=4in]{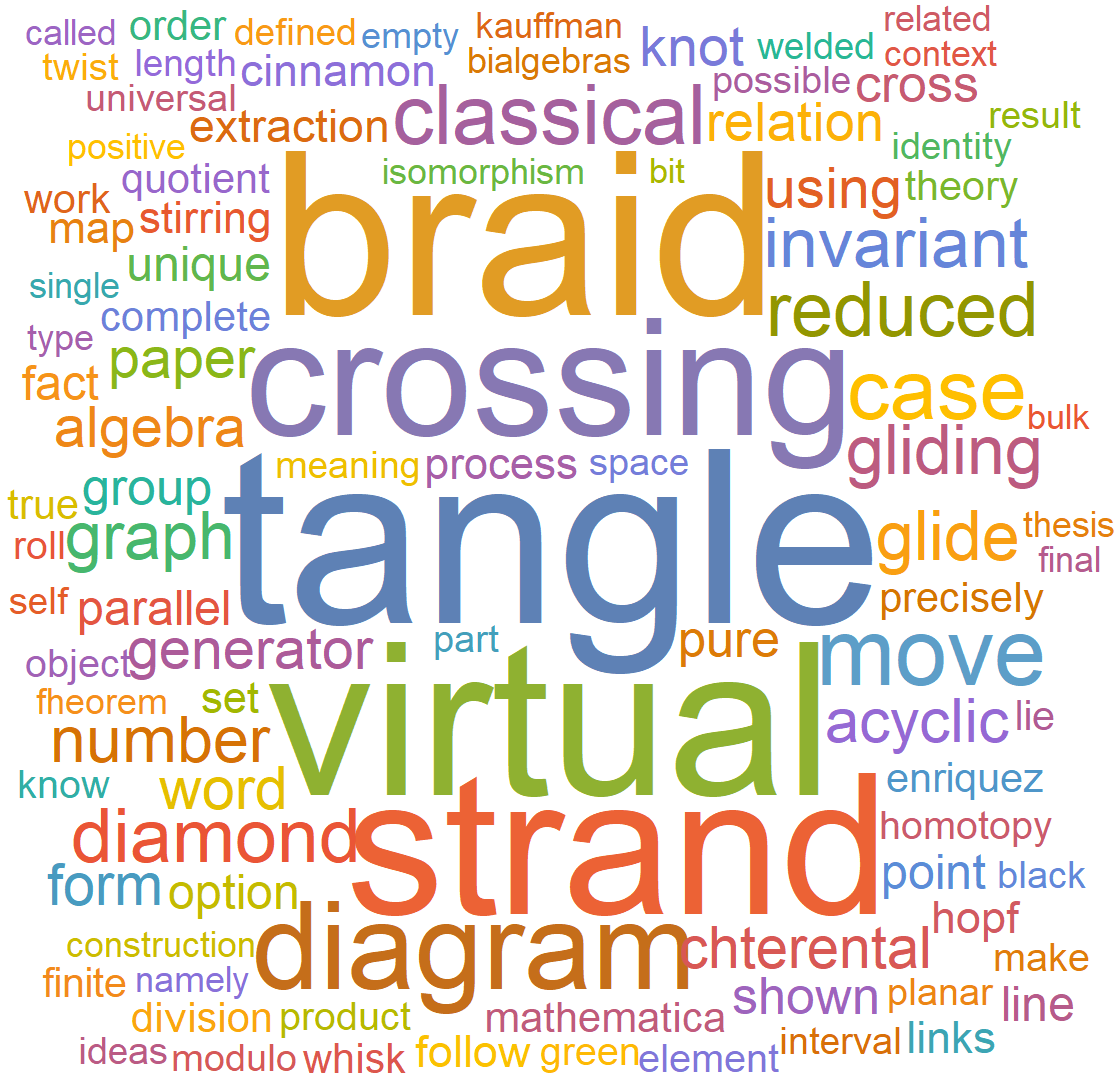} \]
\vfill}

\end{document}